\documentclass[10pt]{article}

\usepackage{amsmath}
\usepackage{amsfonts}
\usepackage{amsthm}
\usepackage[english]{babel}
\usepackage{graphicx}
\usepackage{amssymb}

\newtheorem{theorem}{Theorem}
\newtheorem{lemma}{Lemma}
\newtheorem{definition}{Definition}
\newtheorem{corollary}{Corollary}

\newtheorem{remark}{Remark}
\newtheorem{ex}{Example}

\newcommand{\mR}{\mathbb{R}}

\begin{document}

\title{Explicit formulas for the Dunkl dihedral kernel and the $(\kappa, a)$-generalized Fourier kernel}

\author{Denis Constales$^{1}$ \footnote{E-mail: {\tt  denis.constales@ugent.be}} \and Hendrik De Bie$^{1}$\footnote{E-mail: {\tt hendrik.debie@ugent.be}} \and Pan Lian$^{2}$\footnote{E-mail: {\tt pan.lian@outlook.com}} }

\vspace{10mm}
\date{\small{1:  Department of Mathematical Analysis\\ Faculty of Engineering and Architecture -- Ghent University}\\
\small{Galglaan 2, 9000 Gent,
Belgium}\\ \vspace{5mm}
\small{2: School of Mathematical Sciences -- Tianjin Normal University}\\
\small{
Binshui West Road 393, Tianjin 300387, P.R. China 
}\vspace{5mm}
}

\maketitle

\begin{abstract}
In this paper, a new method is developed to obtain explicit and integral expressions for the kernel of the
$(\kappa, a)$-generalized Fourier transform for $\kappa =0$. In the case of dihedral groups, this method is also applied to the Dunkl kernel as well as the Dunkl Bessel function.
The method uses the introduction of an auxiliary variable in the series expansion of the kernel, which is subsequently Laplace transformed. The kernel in the Laplace domain takes on a much simpler form, by making use of the Poisson kernel. The inverse Laplace transform can then be computed using the generalized Mittag-Leffler function to obtain integral expressions. In case the parameters involved are integers, explicit formulas are obtained using partial fraction decomposition.

New bounds for the kernel of the $(\kappa, a)$-generalized Fourier transform are obtained as well.

~\\

{\em Keywords}: Dunkl kernel, Generalized Fourier transform, Dihedral group, Bessel function, Poisson kernel
~\\

{\em Mathematics Subject Classification:} 33C52, 43A32, 42B10
\end{abstract}

\tableofcontents

\section{Introduction}

Recently, a lot of attention has been given to various generalizations of the Fourier transform. This paper focusses on two in particular, namely the Dunkl transform \cite{D1, deJ} and the $(\kappa, a)$-generalized Fourier transform \cite{SKO}.
Both transforms depend on a number of parameters, and as such it is an open problem, except for certain special cases, to find concrete formulas for their integral kernels.

Our aim in this paper is to develop a new method for finding explicit expressions as well as integral expressions for these kernels. Explicit expressions can be obtained when some of the arising parameters take on rational or integer values. The integral expressions we will obtain are valid in full generality and are expressed in terms of the generalized Mittag-Leffler function (see \cite{MH} or the subsequent Definition \ref{DefML}).

Essentially our method works as follows. Consider the following series expansion, for $x,y \in \mR^m$
 \[
 K^{m}(x,y)=2^{\lambda}\Gamma(\lambda+1) \sum_{j=0}^{\infty} (-i)^j\frac{\lambda+j}{\lambda}z^{-\lambda}J_{j+\lambda}(z)C^{\lambda}_{j}(\xi)\]
 with $\lambda=(m-2)/2$, $z=|x||y|$,  $\xi=\langle x,y\rangle/z$, $J_{j+\lambda}(z)$ the Bessel function and $C_{j}^{\lambda}(\xi)$  the Gegenbauer polynomial.
It is not so easy to recognize that this is the classical Fourier kernel $e^{-i\langle x, y\rangle}$.

However, when we introduce an auxiliary variable $t$ in the kernel  as follows
\[
K^{m}(x,y,t)=2^{\lambda}\Gamma(\lambda+1) \sum_{j=0}^{\infty}(-i)^j \frac{\lambda+j}{\lambda}z^{-\lambda}J_{j+\lambda}(t z)C^{\lambda}_{j}(\xi)
\]
we can take the Laplace transform in $t$ of $K^{m}(x,y,t)$. Simplifying the result by making use of the Poisson kernel (see subsequent Theorem \ref{pe}) then yields
\[
\mathcal{L}(K^{m}(x,y,t))=\Gamma(\lambda+1)\frac{1}{(s+i \langle x, y\rangle)^{\lambda+1}}.
\]
of which we immediately compute the inverse Laplace transform as
\[
K^{m}(x,y,t) = t^{\frac{m-2}{2}} e^{-i t \langle x, y\rangle}
\]
and the classical Fourier kernel is recovered by putting $t=1$.

We develop this method in full generality for the Dunkl kernel related to dihedral groups, as well as for the $(\kappa, a)$- generalized Fourier transform when $\kappa =0$. The restriction to dihedral groups is necessary, because only then the Poisson kernel for the Dunkl Laplace operator is known, see \cite{DX} or subsequent Theorem \ref{DunklPoisson}.

Let us describe our main results. The Laplace transform of the  $(0,a)$-generalized Fourier transform is obtained in Theorem \ref{laplradial}.
When $a = 2/n$ and $m$ is even, the result is a rational function and we can apply partial fraction decomposition to obtain an explicit expression, see Theorem \ref{th3}.  We prove that the kernel  for $a = 2/n$ is bounded by 1 in Theorem \ref{th7}, for both even and odd dimensions. For arbitrary $a$, the integral expression in terms of the generalized Mittag-Leffler function is given in Theorem \ref{ga}.

The Laplace transform of the Dunkl kernel for dihedral groups is obtained in Theorem \ref{ld1}. Two alternative integral expressions for the Dunkl kernel, again in terms of the generalized Mittag-Leffler function, are given in Theorem \ref{m1} and \ref{m2}.

The paper is organized as follows. After the necessary preliminaries in section \ref{prelim}, we first study the $(\kappa,a)$-generalized Fourier transform for $\kappa =0$ in section \ref{radialFT}. In section \ref{dunklFT} we then study the Dunkl kernel  for dihedral groups. We also show how our methods can be applied to the Dunkl Bessel function.

\section{Preliminaries}
\label{prelim}
In this section, we give a brief overview of the theory of Dunkl operators, the $(\kappa, a)$-generalized Fourier transform and the Laplace transform. Most of these results are taken from \cite{DX}, \cite{rm} and \cite{SKO}.
We use the notation $\langle \cdot,\cdot \rangle$ for the standard inner product on $\mathbb{R}^{m}$ and $|\cdot|$ for the associated norm. For a non-zero vector $\alpha\in \mathbb{R}^{m}$, the reflection $r_{\alpha}$ with respect to the hyperplane orthogonal to $\alpha$ is defined by
\[ r_{\alpha}(x)=x-2\frac{\langle \alpha, x \rangle}{|\alpha|^{2}}\alpha.
\]
A reduced root system $\mathcal{R}$ is a finite set of non-zero vectors in $\mathbb{R}^{m}$ such that $r_{\alpha} \mathcal{R}=\mathcal{R}$ and $\mathbb{R} \alpha \cap \mathcal{R}= \{\pm\alpha\}$ for all $\alpha\in \mathcal{R}$. The finite reflection group generated by $\{r_{\alpha}: \alpha\in \mathcal{R}\}$ is  a subgroup of the orthogonal group $O(m)$ which is  called  a Coxeter group. Three infinite families of root systems  are  $A_{n-1}$, $B_{n}$ and the root system associated to the dihedral groups. We give the latter as an example which will be used later.
\begin{ex} In the Euclidean space $\mathbb{R}^{2}$, let $d\in O(2, \mathbb{R})$ be the rotation over $2\pi/k$  and $e$ the reflection at the $y$-axis. The group $I_{k}$ generated by $d$ and $e$ consists of all orthogonal transformations which preserve a regular $k$-sided polygon centered at the origin. The group $I_{k}$ is a finite reflection group which is usually called  dihedral group.
 \end{ex}
We define the action of $G$ on functions by \[(g\cdot f)(x):=f(g^{-1}\cdot x),\qquad x\in \mathbb{R}^{m}, g\in G.\]
A multiplicity function $\kappa: \mathcal{R} \rightarrow \mathbb{C}$ is a function invariant under the action of $G$. Furthermore, set $\mathcal{R}_{+}:=\{\alpha\in \mathcal{R}:\langle\alpha, \beta \rangle>0\}$ for some $\beta \in \mathbb{R}^{m}$ such that $\langle\alpha, \beta \rangle \neq 0$ for all $\alpha\in \mathcal{R}$. From now on, fix the positive subsystem $\mathcal{R}_{+}$  and the multiplicity function $\kappa$. The Dunkl operator $T_{i}$ associated to $\mathcal{R}_{+}$ and $\kappa$  is then defined by
\[T_{i}f(x)=\frac{\partial f}{\partial x_{i}}+\sum_{\alpha\in \mathcal{R}_{+}}\kappa(\alpha) \alpha_{i}\frac{f(x)-f(r_{\alpha}(x))}{\langle \alpha, x\rangle}, \qquad f\in C^{1}(\mathbb{R}^{m})
\]
where $\alpha_{i}$ is the $i$-th coordinate of $\alpha$. All the $T_{i}$ commute with each other. They reduce to the corresponding partial derivatives when $\kappa=0$. The Dunkl Laplacian $\Delta_{\kappa}$ is then defined as $\Delta_{\kappa}=\sum_{i=1}^{m}T_{i}^2$.
The weight function associated with the root system $\mathcal{R}$ and the multiplicity function $\kappa$ is given by \[\upsilon_{\kappa}(x):=\prod_{\alpha\in \mathcal{R}_{+}}|\langle x,\alpha\rangle|^{2\kappa(\alpha)}.\]
It is $G$-invariant and homogeneous of degree $2\langle \kappa\rangle$ where the index $\langle \kappa \rangle$ of the multiplicity function $\kappa$ is defined as \[\langle \kappa\rangle :=\sum_{\alpha\in \mathcal{R}_{+}}\kappa_{\alpha}=\frac{1}{2}\sum_{\alpha\in \mathcal{R}}\kappa_{\alpha}.\]  We also denote by $\mathcal{H}_{j}(\upsilon_{\kappa})$ the space of Dunkl harmonics of degree $j$, i.e. $\mathcal{H}_{j}(\upsilon_{\kappa})=\mathcal{P}_{j}\cap \rm{ker} \Delta_{\kappa}$ with $\mathcal{P}_{j}$ the space of homogeneous polynomials of degree $j$. There exists a unique linear and homogeneous isomorphism on polynomials which intertwines the algebra of Dunkl operators and the algebra of usual partial differential operators, i.e. $V_{\kappa}(\mathcal{P}_{j})=\mathcal{P}_{j}, \quad V_{\kappa}|_{\mathcal{P}_{0}}=id$ and $T_{\xi} V_{\kappa}=V_{\kappa} \partial_{\xi}$ for all $\xi \in \mathbb{R}^{m}$.
In the following, we denote by $P_{j}(G; x,y)$ the reproducing kernel of $\mathcal{H}_{j}(\upsilon_{\kappa})$ and $P(G; x,y)$ the Poisson kernel. For $j \in \mathbb{N}$ and $|y|\le |x|=1$, we have \cite{DX}
\begin{eqnarray}\label{pd} P_{j}(G; x,y)=\frac{j+\lambda_{\kappa}}{\lambda_{\kappa}} V_{\kappa}[C_{j}^{\lambda_{\kappa}}(\langle \cdot, \frac{y}{|y|}\rangle)](x)|y|^{j},\end{eqnarray}
and
\begin{eqnarray}\label{pois}P(G;x,y)=\sum_{j=0}^{\infty} P_{j}(G;x, y)=\sum_{j=0}^{\infty} P_{j}(G;x,\frac{y}{|y|})|y|^{j}=V_{\kappa}\biggl(\frac{1-|y|^2}{(1-2\langle \cdot, y\rangle +|y|^2)^{\lambda_{\kappa}+1}}\biggr)(x)\end{eqnarray}
where $\lambda_{\kappa}= \langle \kappa\rangle +\frac{m-2}{2}$. R\"osler  \cite{rm1} proved there exists a unique positive probability-measure $\mu_{x}(\xi)$ on $\mathbb{R}^{m}$ such that
\[V_{\kappa}f(x)=\int_{\mathbb{R}^{m}}f(\xi)d\mu_{x}(\xi)\] for the positive multiplicity function. In \cite{SKO}, Dunkl's interwining operator $V_{\kappa}$ was extended to $C(B)$ with $B$ the closed unit ball in $\mathbb{R}^m$ for the regular values of $\kappa$.
Denoting \[\tilde{(V_{\kappa}}h):=(V_{\kappa}h_{y})(x)=\int_{\mathbb{R}^{m}}h(\langle \xi, y\rangle)d\mu_{x}(\xi),\]
this operator  satisfies
\begin{eqnarray}\label{sv1}||\tilde{V_{\kappa}}h||_{L^{\infty}(B\times B)}\le ||h||_{L^{\infty}([-1,1])}.\end{eqnarray}

It is known that the operators $T_{j}$ have a joint eigenfunction $E_{\kappa}(x,y)$ satisfying
\[ T_{j}E_{\kappa}(x,y)=-iy_{j}E_{\kappa}(x,y), \qquad j=1,\ldots,m.\]
The function $E_{\kappa}(x,y)$ is called the Dunkl kernel,  which is the exponential function $e^{-i\langle x, y\rangle}$ when $\kappa=0$. This kernel together with the weight function is used to define the so-called Dunkl transform \[\mathcal{F}_{\kappa}: L^{1}(\mathbb{R}^{m}, \upsilon_{\kappa})\rightarrow C(\mathbb{R}^{m})\] by
\[\mathcal{F}_{\kappa}f(y):=c_{\kappa}\int_{\mathbb{R}^{m}}f(x)E(x,y)\upsilon_{\kappa}(x)dx \quad(y\in \mathbb{R}^{m})
\] with $c_{\kappa}^{-1}=\int_{\mathbb{R}^{m}} e^{-|x|^{2}/2} \upsilon_{\kappa}(x)dx$ the Mehta constant associated to $G$. Again, when $\kappa=0$, we recover the classical Fourier transform. The Dunkl transform shares many properties with the Fourier transform. In \cite{HR}, using the harmonic oscillator $-(\Delta-|x|^2)/2$,  Howe found the spectral description of the Fourier transform and its eigenfunctions forming the basis of $L^{2}(\mathbb{R}^{m})$:
\[\mathcal{F}=e^{\frac{i\pi m}{4}}e^{\frac{i\pi}{4}(\Delta-|x|^2)}\]
with $\Delta$ the Laplace operator. Similarly, the Dunkl transform also has the exponential notation
\[\mathcal{F}_{\kappa}=e^{\frac{i\pi \mu}{4}}e^{\frac{i\pi}{4}(\Delta_{\kappa}-|x|^2)}\]
where $\mu=m+2\langle \kappa \rangle$, see \cite{Said}.
In \cite{SKO}, the authors defined a radially deformed Dunkl-type harmonic oscillator \[ \Delta_{\kappa,a}=|x|^{2-a}\Delta_{\kappa}-|x|^{a},\qquad a>0.\]
Then the $(\kappa,a)$-generalized Fourier transform is defined by
\[\mathcal{F}_{\kappa,a}=e^{\frac{i\pi}{2a}(m-2+2\langle \kappa\rangle+a)}e^{\frac{i\pi}{2a} \Delta_{\kappa,a}}\]
in $L^{2}(\mathbb{R}^{m},  |x|^{a-2}\upsilon_{\kappa}(x) )$. We write the $(\kappa,a)$-generalized Fourier transform as an integral transform
\[\mathcal{F}_{\kappa,a}f(y)=c_{\kappa,a}\int_{\mathbb{R}^{m}}B_{\kappa,a}(x,y)f(x)|x|^{a-2}\upsilon_{\kappa}(x)dx\]
where $c_{\kappa,\alpha}^{-1}=\int_{\mathbb{R}^{m}} e^{-|x|^{a}/a}|x|^{a-2}\upsilon_{\kappa}(x)dx$.
The series expression of $B_{\kappa,a}(x,y)$ is given in \cite{SKO} as follows,
\begin{theorem} \label{the1} For $x, y\in \mathbb{R}^{m}$ and $a>0$, we have
\[B_{\kappa,a}(x,y)=a^{\frac{2\langle \kappa\rangle+m-2}{a}}\Gamma\biggl(\frac{2\langle \kappa\rangle+m+a-2}{a}\biggr) \sum_{j=0 }^{\infty}  B_{\kappa,a}^{(j)}(z) P_{j}(G;\omega,\eta)\]
where $x=|x| \omega,$ $y=|y|\eta$, $z=|x||y|$, $\lambda_{\kappa,a,j}=\frac{2j+2\langle \kappa\rangle+m-2}{a}$,
\[B_{\kappa, a}^{(j)}(z)=e^{-i\frac{\pi}{2}\frac{j}{a}} z^{-\langle \kappa\rangle-m/2+1}J_{\lambda_{\kappa,a,j}}\biggl(\frac{2}{a}z^{a/2}\biggr),\]
and
\[P_{j}(G;\omega,\eta):=\biggl(\frac{\langle \kappa\rangle+j+\frac{m-2}{2}}{\langle \kappa\rangle+\frac{m-2}{2}}\biggr)V_{\kappa}[C_{j}^{\lambda_{\kappa}}(\langle \cdot, \eta\rangle)](\omega),
\]
is the reproducing kernel of the space of spherical $\kappa$-harmonic polynomials of degree $j$.
\end{theorem}
This transform recovers the Dunkl transform when $a=2$, the Fourier transform when $a=2$ and $\kappa=0$. The operator $\mathcal{F}_{0,1}$ is the unitary inversion operator of the Schr\"{o}dinger model of the minimal representation of the group $O(m+1,2)$ \cite{KM}. The explicit expression of the Dunkl kernel is only known for the groups $\mathbb{Z}_{2}^{m}$, the root systems $A_{2}$, $B_{2}$ and some  dihedral groups with integer muliplicity function $\kappa$, see\cite{amr}, \cite{Ade},  \cite{D1}, \cite{DX} and \cite{DDY}. For the integral kernel $B_{\kappa,a}(x,y)$, except the already known Dunkl kernel, closed expressions have been found when $\kappa=0$ and $a=\frac{2}{n}$ with $ n\in \mathbb{N}$  in dimension 2, see \cite{Rad2}. For higher even dimension, an iterative procedure using derivatives is given there as well.  Pitt's inequalities and uncertainty principles for the $(\kappa, a)$-generalized Fourier transform have been established in \cite{J1, Go} .

The Laplace transform is an integral transform which takes a function of a positive real variable $t$ to a function of a complex variable $s$. For a function $f(t)$ which has an exponential growth  $|f(t)|\le Ce^{\alpha t}, t\ge t_{0}$, the Laplace transform is defined as
\[F(s)=\mathcal{L}(f(t))(s)=\int_{0}^{\infty}e^{-st}f(t)dt.\]
The inverse Laplace transform is given by the Bromwich integral or the Post's inversion formula. In practice, it is typically more convenient to decompose a Laplace transform into known transforms of functions obtained from a table, for example \cite{E2}. For more details  on the Laplace transform, we refer to \cite{DG}.

\section{The kernel of the $(\kappa,a)$-generalized Fourier transform}
\label{radialFT}
\subsection{Explicit expression of the kernel when $a=\frac{2}{n}$ and $m$ even}
 In this section,  we first establish the connection between the kernel of the $(0,a)$-generalized Fourier kernel and the Poisson kernel for the unit ball by introducing an auxiliary variable in the  kernel and using the Laplace transform. Then we give the explicit formula for the kernel when $a=\frac{2}{n}$ and $m$ even.

The kernel $K^{m}_{a}(x,y)=B_{0, a}(x,y)$  for $a>0$ is given in Theorem \ref{the1} (see also \cite{Rad2}, \cite{SKO}) \[K_{a}^{m}(x,y)=a^{2\lambda/a}\Gamma\biggl(\frac{2\lambda+a}{a}\biggr)\sum_{j=0}^{\infty}e^{-\frac{i\pi j}{a}}\frac{\lambda+j}{\lambda}z^{-\lambda}J_{\frac{2(j+\lambda)}{a}}\biggl(\frac{2}{a}z^{a/2}\biggr)C^{\lambda}_{j}(\xi)\] with $\lambda=(m-2)/2$, $z=|x||y|$,  $\xi=\langle x,y\rangle/z$ and $C_{j}^{\lambda}(\xi)$  the Gegenbauer polynomial.
We introduce an auxiliary variable $t$ in the kernel  as follows \begin{eqnarray}\label{ke1}K_{a}^{m}(x,y,t)=a^{2\lambda/a}\Gamma\biggl(\frac{2\lambda+a}{a}\biggr)\sum_{j=0}^{\infty}e^{-\frac{i\pi j}{a}}\frac{\lambda+j}{\lambda}z^{-\lambda}J_{\frac{2(j+\lambda)}{a}}\biggl(\frac{2}{a}z^{a/2} t\biggr)C^{\lambda}_{j}(\xi).\end{eqnarray}
Before we take the Laplace transform, we give the expansion of the Poisson kernel in terms of  Gegenbauer polynomials.
\begin{theorem} \label{pe}\cite{DX}
For $x,y \in \mathbb{R}^{m}$ and $ |y|\le |x|=1$, the Poisson kernel for the unit ball is
\begin{eqnarray*} P(x,y)=\frac{1-|y|^2}{|x-y|^{m}}=\frac{1-|y|^2}{(1-2\xi|y|+|y|^2)^{m/2}}=\sum_{j=0}^{\infty}\frac{j+m/2-1}{m/2-1}C_{j}^{m/2-1}(\xi)|y|^{j}, \quad \xi=\langle x, \frac{y}{|y|}\rangle. \end{eqnarray*}
\end{theorem}

        This result can be extended for $\lambda>0,$ we have
\begin{eqnarray} \label{ac1}\frac{1-|y|^2}{(1-2\xi|y|+|y|^2)^{\lambda+1}}=\sum_{j=0}^{\infty}\frac{j+\lambda}{\lambda}C_{j}^{\lambda}(\xi)|y|^{j}. \end{eqnarray} It is still valid for $z\in \mathbb{C}$, $|z|<1$ and $|\xi|<1$, (see \cite{olf})
\begin{eqnarray}\label{ac2}\frac{1-z^2}{(1-2\xi z+z^2)^{\lambda+1}}=\sum_{j=0}^{\infty}\frac{j+\lambda}{\lambda}C_{j}^{\lambda}(\xi)z^{j}. \end{eqnarray}
To establish the validity of the analytic continuation of (\ref{ac1}) to (\ref{ac2}), note that the left-hand side of (\ref{ac2}) is analytic in $z$  in any disk centered at the origin of the complex plane that does not contain any zero of the denominator, hence analytic in $0\le |z|<1$. By the estimate \[|C_{j}^{\lambda}(\xi)|\le C_{j}^{\lambda}(1)=\frac{(2\lambda)_{j}}{j!},\]
the right-hand side of (\ref{ac2}) will certainly converge to an analytic continuation of that of (\ref{ac1}) for all $z$ satisfying $|z|\le |y|<1$, hence for the whole unit disk.

By Theorem \ref{pe} and the  formula from \cite{E2} \begin{eqnarray}\label{l1} \mathcal{L}(J_{\nu}(bt))=\frac{1}{\sqrt{s^2+b^2}} \biggl(\frac{b}{s+\sqrt{s^2+b^2}}\biggr)^{\nu}, \qquad \mbox{Re}\,\nu >-1,  \mbox{Re} \, s>|\mbox{Im}\, b|
,\end{eqnarray} we take the Laplace transform with respect to $t$ in (\ref{ke1}).  With $u_{R}=e^{\frac{-i\pi}{a}} (\frac{2z^{a/2}}{aR})^{2/a}$, $r=\sqrt{s^{2}+(\frac{2}{a}z^{a/2})^{2}}$, $R=s+r$, $\lambda=(m-2)/2$, $z=|x||y|$ and $\xi=\langle x,y\rangle/z$, for ${\rm Re} \, s$ big enough, we obtain
\begin{eqnarray} \label{rl}
&&\mathcal{L}(K_{a}^{m}(x,y,t))\nonumber\\&=&2^{2\lambda/a}\Gamma\biggl(\frac{2\lambda+a}{a}\biggr)\frac{1}{r}\biggl(\frac{1}{R}\biggr)^{2\lambda/a}\frac{1-u_{R}^{2}}{(1-2\xi u_{R}+u_{R}^2)^{\lambda+1}} \nonumber\\
&=&2^{2\lambda/a}\Gamma\biggl(\frac{2\lambda+a}{a}\biggr)\frac{1}{r}\frac{R^{2/a}-\frac{e^{-2i\pi/a}(2/a)^{4/a}z^{2}}{R^{2/a}}}{(R^{2/a}-2\xi e^{-i\pi/a}(2/a)^{2/a}z+\frac{e^{-2i\pi/a}(2/a)^{4/a}z^{2}}{R^{2/a}} )^{\lambda+1}}
\nonumber
\\&=&2^{2\lambda/a}\Gamma\biggl(\frac{2\lambda+a}{a}\biggr)\frac{1}{r}\frac{(s+r)^{2/a}-e^{-2i\pi/a}(r-s)^{2/a}}{((s+r)^{2/a}-2\xi e^{-i\pi/a}(2/a)^{2/a}z+e^{-2i\pi/a}(r-s)^{2/a})^{\lambda+1}}.
\end{eqnarray}

The validity of  transforming term by term in (\ref{ke1}) is guaranteed by the following theorem.
\begin{theorem} \cite{DG}
Let the function $F(s)$ be represented by a series of $\mathcal{L}$-transforms \[F(s)=\sum_{v=0}^{\infty}F_{v}(s), \quad F_{v}(s)=\mathcal{L}(f_{v}(t)),\]
where all integrals
\[\mathcal{L}(f_{v})=\int_{0}^{\infty}e^{-st}f_{v}(t)dt=F_{v}(s), \quad(v=0,1,\cdots)\]
converge in a common half-plane ${\rm Re}\,s \ge x_{0}$. Moreover, we require that
the integrals \[\mathcal{L}(|f_{v}|)=\int_{0}^{\infty}e^{-st}|f_{v}(t)|dt=G_{v}, \quad(v=0,1,\cdots)\]
and the series \[\sum_{v=0}^{\infty}G_{v}(x_{0})\]converge which implies that $\sum_{v=0}^{\infty}F_{v}(s)$ converges absolutely and uniformly in ${\rm Re}  \,s\ge x_{0}$.
Then $\sum_{v=0}^{\infty}f_{v}(t)$ converges, absolutely, towards a function $f(t)$ for almost all $t\ge 0$; this f(t) is the original function of $F(s)$;
\[\mathcal{L}\biggl(\sum_{v=0}^{\infty}f_{v}(t)\biggr)=\sum_{v=0}^{\infty}F_{v}(s).\]
\end{theorem}

Hence we can summarize our results as follows,
\begin{theorem}
\label{laplradial}
The kernel of the deformed Fourier transform in the Laplace domain is
\begin{eqnarray}\label{rl2}
&&\mathcal{L}(K_{a}^{m}(x,y,t))\nonumber
\nonumber
\\&=&2^{2\lambda/a}\Gamma\biggl(\frac{2\lambda+a}{a}\biggr)\frac{1}{r}\frac{(s+r)^{2/a}-e^{-2i\pi/a}(r-s)^{2/a}}{((s+r)^{2/a}-2\xi e^{-i\pi/a}(2/a)^{2/a}z+e^{-2i\pi/a}(r-s)^{2/a})^{\lambda+1}}
\end{eqnarray}
where $r=\sqrt{s^{2}+(\frac{2}{a}z^{a/2})^{2}}.$
\end{theorem}
By direct computation, we have the following simpler expression when $m > 2$.
\begin{corollary} When $\lambda>0$, the kernel of the deformed Fourier transform in the Laplace domain is
\begin{eqnarray*}
&&\mathcal{L}(K_{a}^{m}(x,y,t))\nonumber\nonumber
\\&=&-2^{2\lambda/a}\Gamma\biggl(\frac{2\lambda}{a}\biggr)\frac{d}{d s}\biggl(\frac{1}{((s+r)^{2/a}-2\xi e^{-i\pi/a}(2/a)^{2/a}z+e^{-2i\pi/a}(r-s)^{2/a})^{\lambda}}\biggr)
\end{eqnarray*}
where $r=\sqrt{s^{2}+(\frac{2}{a}z^{a/2})^{2}}.$
\end{corollary}
Let us now look at a few special cases.
When $a=1$, (\ref{rl2}) reduces to
\[\mathcal{L}(K_{1}^{m}(x,y,t))=\Gamma(2\lambda+1)\frac{s}{(s^2+2z+2\xi z)^{\lambda+1}}.\]
Using the formula in \cite{E2} \begin{eqnarray}\label{nf1}\mathcal{L}^{-1}(2^{\nu+1}\pi^{-1/2}\Gamma(\nu+3/2)a^{\nu}\sqrt{s^2+a^2}^{-2\nu-3}s )=t^{\nu+1}J_{\nu}(at),\qquad \mbox{Re}\,\nu >-1,  \mbox{Re}\, s>|\mbox{Im} \, a| \end{eqnarray} and then setting $t=1$ in $K_{1}^{m}(x,y,t)$,
we reobtain the kernel \[K_{1}^{m}(x,y)=\Gamma(\lambda+1/2)\tilde{J}_{\frac{m-3}{2}}(\sqrt{2(|x||y|+\langle x,y\rangle)}) \]
with $\tilde{J}_{\nu}(z)=J_{\nu}(z)(z/2)^{-\nu}$, see \cite{Rad3}.

When $a=2$, (\ref{rl2}) reduces to \[ \mathcal{L}(K_{2}^{m}(x,y,t))=\Gamma(\lambda+1)\frac{1}{(s+i\xi z)^{\lambda+1}}.
\]
By the inverse transform formula in \cite{E2} \begin{eqnarray*} \mathcal{L}\biggl(\frac{t^{k-1}e^{-\alpha t}}{\Gamma(k)}\biggr)=\frac{1}{(s+\alpha)^{k}} \qquad k>0,\end{eqnarray*}
and then putting $t=1$ in $K_{2}^{m}(x,y,t)$,
we get the classical Fourier kernel \[K_{2}^{m}(x,y)=e^{-i\langle x, y\rangle}.\]

We are interested in the case when $a=\frac{2}{n}$, because it has a close relationship with the Dunkl kernel and Dunkl Bessel function associated with dihedral groups which we will discuss in Section 4. When $a=\frac{2}{n}$,  the Fourier kernel in the Laplace domain is  \begin{eqnarray} \label{rl3} \mathcal{L}(K_{\frac{2}{n}}^{m}(x,y,t))= \Gamma(n\lambda+1)\frac{Q_{n-1}(s)}{P_{n}(s)^{\lambda+1}} ,\end{eqnarray}
with \begin{eqnarray*}Q_{n-1}(s)&=&\frac{(s+r)^{n}-e^{-in\pi}(r-s)^{n}}{2^{n}r},\\ P_{n}(s)&=&\frac{(s+r)^{n}-2\xi e^{-in\pi/2}(n)^{n}z+e^{-in\pi}(r-s)^{n}}{2^{n}}.\end{eqnarray*}
By direct computation, we have
 \begin{eqnarray}\label{dr1}\frac{d }{d s}P_{n}(s)=nQ_{n-1}(s),\end{eqnarray}
and \begin{eqnarray}\label{dl1}  \mathcal{L}(K_{\frac{2}{n}}^{m}(x,y,t))= \Gamma(n\lambda+1)\frac{\frac{d }{d s}P_{n}(s)}{n(P_{n}(s))^{\lambda+1}}
=-\Gamma(n\lambda)\frac{d}{ds}\frac{1}{P_{n}(s)^{\lambda}} ,\end{eqnarray}
when $\lambda>0$.

We can investigate both functions $Q_{n-1}(s)$ and $P_{n}(s)$ in more detail. This is done in the following lemma.
\begin{lemma}\label{lem1} The function $P_{n}(s)$ is a polynomial of degree $n$ in $s$ with the factorization
\[P_{n}(s)=\prod_{l=0}^{n-1}\biggl(s+inz^{1/n} \cos\biggl(\frac{q+2\pi l}{n}\biggr)\biggr),
\]
where $q=\arccos(\xi)$, $\xi=\frac{\langle x, y\rangle}{|x||y|}$.
 The function $Q_{n-1}(s)$ is a polynomial of degree $n-1$ in $s$. When $n$ is odd, $Q_{n-1}(s)$ has the factorization
\[Q_{n-1}(s)=\prod_{l=1}^{n-1}\biggl(s-inz^{1/n}\cos\biggl(\frac{l\pi}{n}\biggr)\biggl).\]
When $n$ is even, $Q_{n-1}(s)$ has the factorization\[Q_{n-1}(s)=\prod_{l=0,l\neq\frac{n}{2}}^{n-1}\biggl(s-inz^{1/n}\sin\biggl(\frac{l\pi}{n}\biggr)\biggl).\]
\end{lemma}
\begin{proof}
\begin{enumerate}
  \item We show that $P_{n}(s)$ is a polynomial of degree $n$ in $s$,
  \begin{eqnarray*}2^{n}P_{n}(s)&=&(s+r)^{n}-2\xi e^{-in\pi/2}(n)^{n}z+e^{-in\pi}(r-s)^{n}\\&=& (s+r)^{n}+(-1)^{n}(r-s)^{n}-2\xi e^{-in\pi/2}(n)^{n}z\\&=&\sum_{k=0}^{n}\binom{n}{k}s^{n-k}r^{k}+(-1)^{n}\sum_{k=0}^{n}\binom{n}{k}(-1)^{n-k}s^{n-k}r^{k}-2\xi e^{-in\pi/2}(n)^{n}z\\&=&\biggl(\sum_{k=0}^{n}\binom{n}{k}s^{n-k}r^{k}(1+(-1)^{k})\biggr)-2\xi e^{-in\pi/2}(n)^{n}z
  \\&=&2\sum_{k=0}^{ \lfloor n/2  \rfloor}\binom{n}{2k}s^{n-2k}(s^2+(nz^{1/n})^2)^{k}-2\xi e^{-in\pi/2}(n)^{n}z.
  \end{eqnarray*} Hence $2^{n}P_{n}(s)$ is a polynomial of degree $n$ in $s$. The coefficient of $s^{n}$ is $2\sum_{k=0}^{ \lfloor n/2  \rfloor}\binom{n}{2k}=2^{n}.$
  \item We verify $2^{n}P_{n}(s_{l})=0$ with $s_{l}=-inz^{1/n} \cos(\frac{q+2\pi l}{n}),\quad l=0,\cdots, n-1$. Denote $\xi=\cos (q)=\frac{e^{iq}+e^{-iq}}{2}$.  When $\sin(\frac{q+2\pi l}{n})\ge 0$, we have
      \begin{eqnarray*}2^{n}P_{n}(s_{l})&=&(-inz^{1/n})^{n}\biggl[\biggl(\cos\biggl(\frac{q+2\pi l}{n}\biggr)+i\sin\biggl(\frac{q+2\pi l}{n}\biggr)\biggr)^{n}-2\xi\\&&+\biggl(\cos\biggl(\frac{q+2\pi l}{n}\biggr)-i\sin\biggl(\frac{q+2\pi l}{n}\biggr)\biggr)^{n}\biggr]\\&=&(-inz^{1/n})^{n}\biggl(e^{iq}-2\biggl(\frac{e^{iq}+e^{-iq}}{2}\biggr)+e^{-iq}\biggr)\\&=&0.
      \end{eqnarray*} Similarly,  we have $2^{n}P_{n}(s_{l})=0$ when $\sin(\frac{q+2\pi l}{n})<0$. Hence, $s_{l}, l=0,\cdots, n-1$ are all roots of $2^{n}P_{n}$ and we get the factorization
  \[P_{n}(s)=\prod_{l=0}^{n-1}\biggl(s+inz^{1/n} \cos\biggl(\frac{q+2\pi l}{n}\biggr)\biggr)
.\]
  \item For $2^{n}Q_{n-1}(s)$, we have \begin{eqnarray*} 2^{n}Q_{n-1}(s)&=&\frac{(s+r)^{n}-e^{-in\pi}(r-s)^{n}}{r}\\
  &=&\frac{1}{r}((s+r)^{n}-(-1)^{n}(r-s)^{n})
  \\&=&\frac{1}{r}\sum_{k=0}^{n}\binom{n}{k}s^{n-k}r^{k}(1-(-1)^{n}(-1)^{n-k})
  \\&=&\frac{2}{r}\sum_{k=0}^{\lfloor n/2 \rfloor }\binom{n}{2k+1}s^{n-2k-1}r^{2k+1}
  \\&=&2\sum_{k=0}^{\lfloor n/2 \rfloor }\binom{n}{2k+1}s^{n-2k-1}(s^{2}+(nz^{1/n})^{2})^{k}.
   \end{eqnarray*} So $2^{n}Q_{n-1}(s)$ is a polynomial of degree $n-1$ in $s$.
   \item
When $n$ is odd,  $s_{l}=inz^{1/n}\cos(\frac{l\pi}{n})=inz^{1/n}\sin(\frac{\pi}{2}+\frac{l\pi}{n})$, $l=0,\cdots, n-1$ are $n$ roots of $(2^{n}rQ_{n-1})(s)=0$. Indeed, we have $r_{l}=\sqrt{s_{l}^2+(nz^{1/n})^2}=-nz^{1/n}\cos(\frac{\pi}{2}+\frac{l\pi}{n})$ and
\begin{eqnarray*}
2^{n}r_{l}Q_{n-1}(s_{l})&=&(s_{l}+r_{l})^{n}-e^{-in\pi}(r_{l}-s_{l})^{n}\\&=&(s_{l}+r_{l})^{n}+(r_{l}-s_{l})^{n}\\&=&(-nz^{1/n})^{n}(e^{-i\frac{\pi n}{2}-il\pi}+e^{i\frac{\pi n}{2}+il\pi})\\&=&0\end{eqnarray*}
because $n$ is odd. Note that  $r_{l}= 0$ if and only if when $l=0$. So $s_{l}$, $l=1,\cdots, n-1$ are  the $n-1$ roots of the polynomial $Q_{n-1}(s)$.
Hence, we have \begin{eqnarray*}Q_{n-1}(s)&=&\prod_{l=1}^{n-1}\biggl(s-inz^{1/n}\sin\biggl(\frac{\pi}{2}+\frac{l\pi}{n}\biggr)\biggl).
\end{eqnarray*}
When $n$ is even, we verify $2^{n}r_{l}Q_{n-1}(s_{l})=0$ with $s_{l}=inz^{1/n}\sin(\frac{l\pi}{n})$, $l=0, \cdots, n-1$. For $l\le \frac{n}{2}$,\begin{eqnarray*}
2^{n}r_{l}Q_{n-1}(s_{l})&=&(s_{l}+r_{l})^{n}-e^{-in\pi}(r_{l}-s_{l})^{n}\\&=&(s_{l}+r_{l})^{n}-(r_{l}-s_{l})^{n}\\&=&(nz^{1/n})^{n}(e^{il\pi}-e^{-il\pi})\\&=&0.\end{eqnarray*}
Similarly, for $l>\frac{n}{2}$, we have $2^{n}r_{l}Q_{n-1}(s_{l})=0$. Moreover, we have $r_{l}=0$ if and only if $l=\frac{n}{2}.$ So $s_{l}$, $l\neq\frac{n}{2}$ are the $n-1$ roots of the polynomial $Q_{n-1}(s)$. Hence we have
\begin{eqnarray*}Q_{n-1}(s)&=&\prod_{l=0,l\neq\frac{n}{2}}^{n-1}\biggl(s-inz^{1/n}\sin\biggl(\frac{l\pi}{n}\biggr)\biggl).\end{eqnarray*}

\end{enumerate}
\end{proof}
We now have all the tools necessary to compute the inverse Laplace transform. First we treat the case of dimension 2.
\begin{theorem} \label{co} For $a=\frac{2}{n},  n\in \mathbb{N}$ and $m=2$, we have
\[K_{\frac{2}{n}}^{2}(x,y)=\frac{1}{n}\sum_{l=0}^{n-1}e^{-inz^{1/n} \cos(\frac{q+2\pi l}{n})}.\]
\end{theorem}
\begin{proof} We have, using (\ref{rl3}) and (\ref{dr1})
\begin{eqnarray*}
\mathcal{L}(K^{2}_{\frac{2}{n}}(x,y,t))=\frac{Q_{n-1}(s)}{P_{n}(s)}=\frac{1}{n}\frac{\frac{d}{d s}P_{n}(s)}{P_{n}(s)}=\frac{1}{n}\sum_{l=0}^{n-1}\frac{1}{s+inz^{1/n} \cos(\frac{q+2\pi l}{n})}
.\end{eqnarray*}
Taking the inverse Laplace transform and putting $t=1$ yields the result.
\end{proof}
\begin{remark} This result was  previously obtained in \cite{Rad2} in a different way, using series multisection.
\end{remark}

When the dimension $m>2$, we first use (\ref{dl1}) to obtain
\begin{eqnarray} \label{kl1} K_{\frac{2}{n}}^{m}(x,y,t)
=-\Gamma(n\lambda)\mathcal{L}^{-1}\biggl(\frac{d}{ds}\frac{1}{P_{n}(s)^{\lambda}}\biggr).\end{eqnarray}
The inverse Laplace transform can be computed using the property of the Laplace transform
\begin{eqnarray*}\mathcal{L}^{-1}\biggl(-\frac{d}{ds}\mathcal{L}(f(t))\biggr)=tf(t) \end{eqnarray*}
and the partial fraction decomposition
\begin{eqnarray}\label{kl2}
\mathcal{L}^{-1}\biggl(\frac{1}{P_{n}(s)^{\lambda}}\biggr)=\sum_{k=1}^{n}\sum_{l=1}^{\lambda} \frac{\Phi_{kl}(a_{k}) }{(\lambda-l)!(l-1)!} t^{\lambda-l}e^{a_{k}t}
\end{eqnarray}
with  $a_{k}=-inz^{1/n} \cos(\frac{q+2\pi k}{n})$, $q=\arccos(\xi)$ and
\[\Phi_{kl}(s)=\frac{d^{l-1}}{d  s^{l-1}}\biggl[\biggl(\frac{s-a_{k}}{P_{n}(s)}\biggr)^{\lambda}\biggr].\]
Putting $t=1$ in (\ref{kl1}) and (\ref{kl2}), then yields

\begin{theorem}\label{th3}
When $a=\frac{2}{n},  n\in \mathbb{N}$, the kernel of the $(0,a)$-generalized Fourier transform in even dimension $m>2$ is given by \[
K_{\frac{2}{n}}^{m}(x,y)=\Gamma(n\lambda)\sum_{k=1}^{n}\sum_{l=1}^{\lambda} \frac{\Phi_{kl}(-inz^{1/n} \cos(\frac{q+2\pi k}{n})) }{(\lambda-l)!(l-1)!}e^{-inz^{1/n} \cos(\frac{q+2\pi k}{n})}.\]
\end{theorem}

As we have given the factored form of  $P_{n}(s)$ in Lemma \ref{lem1}, it is possible to  give an explicit formula of $\Phi_{kl}(s)$ by the following result from \cite{Bo}.
\begin{theorem} \label{pfd} Suppose $\phi(s)$ is a proper rational function having $m$ zeros $ -\sigma_{h}$ of multiplicity $M_{h}$ and $n$ poles $-s_{k}$ of multiplicity $N_{k}$,
\[ \phi(s)=\frac{\prod_{h=1}^{m}(s+\sigma_{h})^{M_{h}}}{\prod_{k=1}^{n}(s+s_{k})^{N_{k}}}.
\] Define the functions
\[f_{k}(s)=\phi(s)(s+s_{k})^{N_{k}}=\frac{\prod_{h=1}^{m}(s+\sigma_{h})^{M_{h}}}{\prod_{\substack{k'=1,\\k'\neq k}}^{n}(s+s_{k'})^{N_{k'}}}, \qquad k=1,2,\cdots, n,\] obtained from $\phi(s)$ by removing the factor $(s+s_{k})^{N_{k}}.$ The first derivative of $f_{k}(s)$ is given by
\[ f_{k}^{(1)}(s)=f_{k}(s)g_{k}(s)
\] with \[g_{k}(s)=\sum_{h=1}^{m}\frac{M_{h}}{s+\sigma_{h}}-\sum_{\substack{k'=1,\\k'\neq k}}^{n}\frac{N_{k'}}{s+s_{k'}} .\] The $r$-th derivative of $g_{k}$ is given by
\[g_{k}^{(r)}(s)=(-1)^{r}r!\biggl[\sum_{h=1}^{m}\frac{M_{h}}{(s+\sigma_{h})^{r+1}} -\sum^{n}_{\substack{k'=1,\\k'\neq k}}\frac{N_{k'}}{(s+s_{k'})^{r+1}}\biggr].
\]
The $i$-th derivative of $f_{k}(s)$ can be expressed by
 \begin{eqnarray*}
 &&f_{k}^{(i)}=(-1)^{i-1}f_{k}^{(0)}\\&&
 \begin{vmatrix}-1 & 0&0& \cdots&0&0&g_{k}^{(0)} \\
  g_{k}^{(0)} & -1&0& \cdots&0&0&g_{k}^{(1)}\\
   2g_{k}^{(1)} & g_{k}^{(0)}&-1& \cdots&0&0&g_{k}^{(2)}\\
   &&&\cdots&&&\\  (i-1)g_{k}^{(i-2)} & \binom{i-1}{2}g_{k}^{(i-3)}&\binom{i-1}{3}g_{k}^{(i-4)}& \cdots&(i-1)g_{k}^{(1)}&g_{k}^{(0)}&g_{k}^{(i-1)} \end{vmatrix}
 .\end{eqnarray*}
\end{theorem}
\subsection{Generating function when $a=\frac{2}{n}$ and $m$ even}
For fixed $a=\frac{2}{n}$ and $n\in \mathbb{N}$, we define the formal generating function of the $(0,a)$-generalized Fourier kernel of even dimension by
\[G_{\frac{2}{n}}(x,y,\varepsilon)=\sum_{\lambda=0}^{\infty}\frac{1}{2^{n\lambda} \Gamma(n\lambda+1)}(-2e^{-in\pi/2}(n)^{n}z\varepsilon)^{\lambda}K_{\frac{2}{n}}^{m}(x,y)
.\]
We introduce an auxiliary variable $t$ in the generating function  as
\[G_{\frac{2}{n}}(x, y, \varepsilon, t)=\sum_{\lambda=0}^{\infty}\frac{1}{2^{n\lambda} \Gamma(n\lambda+1)}(-2e^{-in\pi/2}(n)^{n}z\varepsilon)^{\lambda}K_{\frac{2}{n}}^{m}(x,y,t).\]
 Then we compute the Laplace transform of $G_{\frac{2}{n}}(x, y, \varepsilon, t)$, and get
\begin{eqnarray*}
\mathcal{L}(G_{\frac{2}{n}}(x, y, \varepsilon, t))&=&\sum_{\lambda=0}^{\infty}
\frac{1}{r}\frac{((s+r)^{n}-e^{-in\pi}(r-s)^{n})(-2e^{-in\pi/2}(n)^{n}z\varepsilon)^{\lambda}}{((s+r)^{n}-2\xi e^{-in\pi/2}(n)^{n}z+e^{-in\pi}(r-s)^{n})^{\lambda+1}}
\\&=&\frac{1}{r}\frac{(s+r)^{n}-e^{-in\pi}(r-s)^{n}}{(s+r)^{n}-2(\xi+\varepsilon)e^{-in\pi/2}(n)^{n}z+e^{-in\pi}(r-s)^{n}}.
\end{eqnarray*}
Comparing  with Theorem \ref{co}, we find the only difference is that $\xi$ in the latter becomes $\xi+\varepsilon$.
Now we can give the generating function by  taking the inverse Laplace transform and setting $t=1$.
\begin{theorem} Let $a=2/n$, with $n\in \mathbb{N}$. Then the formal generating function of the $(0,a)$-generalized Fourier kernel of even dimension is
\begin{eqnarray*}G_{\frac{2}{n}}(x, y, \varepsilon)&=&\sum_{\lambda=0}^{\infty}\frac{1}{2^{n\lambda} \Gamma(n\lambda+1)}(-2e^{-in\pi}(n)^{n}z\varepsilon)^{\lambda}K_{\frac{2}{n}}^{m}(x,y)\\&=&\frac{1}{n}\sum_{l=0}^{n-1}e^{-inz^{1/n} \cos(\frac{\tilde{q}+2\pi l}{n})},\end{eqnarray*}
with $\tilde{q}=\arccos (\xi+\varepsilon)$.\end{theorem}
\begin{remark} By taking  consecutive derivatives with respect to $\varepsilon$, we can get an alternative expression for the even dimensional kernel $K_{\frac{2}{n}}^{m}(x,y)$. This coincides with Proposition 2 in \cite{Rad2} and Theorem 1 in \cite{Dn}.
\end{remark}

\subsection{The bounds of the kernel when $a=\frac{2}{n}$ and $m\ge 2$}
In this section, we prove the boundedness of the kernel $K_{\frac{2}{n}}^{m}(x,y)$, $m\ge 2$. This is not obvious from the explicit expansion in Theorem \ref{th3} as we don't know the bounds of $\Phi_{kl}(a_{k})$ in (\ref{kl2}). We first establish a technical lemma. Let us recall the convolution formula of the Laplace transform. Denoting $\mathcal{L}(g(t))=G(s)$ and $\mathcal{L}(f(t))=F(s)$, we have
\begin{eqnarray} \label{lc}\mathcal{L}^{-1}(G(s)F(s))&=&\int_{0}^{t}g(t-\tau)f(\tau)d\tau.\end{eqnarray}
\begin{lemma}\label{le1} For $a_{j}\in \mathbb{R}, j=1,\cdots,n,$ and $ k>0$, put \[F_{n, k}(s)=\displaystyle\frac{1}{\prod_{j=1}^{n}(s+ia_{j})^{k}}\] with inverse Laplace transform \[f_{n, k}(t)=\mathcal{L}^{-1}(F_{n,k}(s)).\]Then
\[|f_{n, k}(t)|\le \frac{t^{nk-1}}{\Gamma(nk)}, \qquad\forall t\in ]0,\infty[.\]
\end{lemma}
\begin{proof} We prove it by  induction.  By the Laplace transform formula \begin{eqnarray*} \mathcal{L}\biggl(\frac{t^{k-1}e^{-\alpha t}}{\Gamma(k)}\biggr)=\frac{1}{(s+\alpha)^{k}}, \qquad  k>0,\end{eqnarray*} we have \[f_{1, k}(t)=\frac{t^{k-1}}{\Gamma(k)}e^{-ia_{1}t},\] so \[|f_{1, k}(t)|\le \frac{t^{k-1}}{\Gamma(k)}.\]
When $n=2$, by the convolution formula (\ref{lc}), we have \begin{eqnarray*} |f_{2, k}(t)|&=&\biggl|\int_{0}^{t}\frac{(t-\tau)^{k-1}e^{-ia_{1}(t-\tau)}}{\Gamma(k)} f_{1, k}(\tau)d\tau\biggr|\le\int_{0}^{t}\frac{(t-\tau)^{k-1}}{\Gamma(k)}|f_{1, k}(\tau)|d\tau\\
&\le&\frac{1}{\Gamma(k)^{2}}\int_{0}^{t}(t-\tau)^{k-1}\tau^{k-1}d\tau
=\frac{t^{2k-1}}{\Gamma(k)^{2}}\int_{0}^{1}(1-x)^{k-1}x^{k-1}dx\\&=&\frac{t^{2k-1}}{\Gamma(2k)}
\end{eqnarray*}
where we have substituted $\tau=tx$ in the third integral.
We assume \begin{eqnarray}\label{mi1}|f_{n-1, k}(t)|\le \frac{t^{(n-1)k-1}}{\Gamma((n-1)k)}.\end{eqnarray}
Then by the convolution formula (\ref{lc}) and (\ref{mi1}), we have
\begin{eqnarray*} |f_{n, k}(t)|&\le& \int_{0}^{t}\frac{(t-\tau)^{k-1}e^{-ia_{n}(t-\tau)}}{\Gamma(k)}|f_{n-1}(\tau)|d\tau
\\&\le&\int_{0}^{t}\frac{(t-\tau)^{k-1}}{\Gamma(k)}\frac{\tau^{(n-1)k-1}}{\Gamma((n-1)k)}d\tau
\\&\le&\frac{t^{nk-1}}{\Gamma(k)\Gamma((n-1)k)}\int_{0}^{1}x^{(n-1)k-1}(1-x)^{k-1}dx
\\&\le&t^{nk-1}\frac{B((n-1)k,k)}{\Gamma(k)\Gamma((n-1)k)}\\&=&\frac{t^{nk-1}}{\Gamma(nk)}.
\end{eqnarray*}
where we used the same substitution as before, and with $B(u,v)$ the beta function.
\end{proof}

By (\ref{dl1}), when $\lambda>0$,
\begin{eqnarray}  \label{lk1} \mathcal{L}(K_{\frac{2}{n}}^{m}(x,y,t))=-\Gamma(n\lambda)\frac{d}{ds}\frac{1}{(P_{n}(s))^{\lambda}}
=-\Gamma(n\lambda)\frac{d}{ds}\frac{1}{\biggl(\prod_{l=0}^{n-1}\biggl(s+inz^{1/n} \cos\biggl(\frac{q+2\pi l}{n}\biggr)\biggr)\biggr)^{\lambda}}. \end{eqnarray}
Setting $t=1$, we get
\begin{eqnarray*}K_{\frac{2}{n}}^{m}(x,y)=\Gamma(n\lambda)f_{n, \lambda}(1)\end{eqnarray*} with $a_{l}=nz^{1/n} \cos\biggl(\frac{q+2\pi l}{n}\biggr)$ in $f_{n,\lambda}(t)$. The problem of finding an integral expression of $K_{\frac{2}{n}}^{m}(x,y)$ thus reduces to finding an integral expression of the function $f_{n, \lambda}(t)$.

From the Laplace transform table \cite{E2}, we have
\begin{eqnarray*} \mathcal{L}^{-1}\biggl(\frac{d}{ds}\biggl(\frac{1}{((s+ib)(s-ib))^{\nu+1/2}}\biggr)\biggr)
=\frac{\sqrt{\pi}}{2^{\nu}\Gamma(\nu+1/2)}t^{\nu+1}\frac{J_{\nu}(b t)}{b^{\nu}}, \qquad {\rm Re}\, \nu>-1, {\rm Re}\,s>|{\rm Im}\, b|.
\end{eqnarray*} Compared with (\ref{lk1}), the Fourier kernel $K_{\frac{2}{n}}^{m}(x,y)$ and $f_{n, k}(t)$ could be thought of as a generalization of the Bessel function. We will see similar behavior in the Dunkl case, see Section 4.

By the inverse Laplace formula from \cite{E2},
\begin{eqnarray*}
\mathcal{L}^{-1}\biggl(\frac{\Gamma(\nu)}{(s+a)^{\nu}(s+b)^{\nu}}\biggr)=\sqrt{\pi}\biggl(\frac{t}{a-b}\biggr)
^{\nu-1/2}e^{-\frac{(a+b)t}{2}}I_{\nu-1/2}\biggl(\frac{a-b}{2}t\biggr), \qquad {\rm Re}\,\nu>0.
\end{eqnarray*}
we can express $f_{n, \lambda}(t)$ as the  convolution of Bessel functions and exponential functions, using (\ref{lc}).

In particular, when $n=3$, and ${\rm Re}\,s >0$, we have
\begin{eqnarray} \label{I3} f_{3, k}(t)=\mathcal{L}^{-1}(F_{3, k}(s))&=&\mathcal{L}^{-1}\biggl(\displaystyle\frac{1}{\prod_{j=1}^{3}(s+ia_{j})^{k}}\biggr)
\nonumber\\&=&\frac{t^{3k-1}}{\Gamma(3k)}e^{ia_{1}t}\Phi_{2}(k,k;3k; i(a_{1}-a_{2})t, i(a_{1}-a_{3})t )
\end{eqnarray}
where $\Phi_{2}(c_{1},c_{2};c_{3};w,z)=\sum_{k,l=0}^{\infty}\frac{(c_{1})_{k}(c_{2})_{l}}{(c_3)_{k+l}}\frac{w^{k}z^{l}}{k!l!},$
see \cite{PB}. Another derivation of the expression obtained here without using Laplace transform  is given in \cite{DDa}.

Now  we can give the main result of this subsection,
\begin{theorem}\label{th7} For $n\in \mathbb{N}$ and $m\ge 2$,  the kernel of the $(0,2/n)$-generalized Fourier transform satisfies
\[|K^{m}_{\frac{2}{n}}(x,y)| \le  1.\]
\end{theorem}
\begin{proof} When $a=\frac{2}{n}$, the Laplace transform of the $(0,a)$-generalized Fourier kernel is
\begin{eqnarray*}
&&\mathcal{L}(K_{\frac{2}{n}}^{m}(x,y,t))= \Gamma(n\lambda+1)G_{1}(s)G_{2}(s)
\end{eqnarray*}
with \begin{eqnarray*} &&G_{1}(s)=\frac{Q_{n-1}(s)}{\prod_{l=0}^{n-1}(s+inz^{1/n} \cos(\frac{q+2\pi l}{n})) }, \\ &&G_{2}(s)=\frac{1}{(\prod_{l=0}^{n-1}(s+inz^{1/n} \cos(\frac{q+2\pi l}{n})))^{\lambda}}.\end{eqnarray*} Denote $g_{j}(t)=\mathcal{L}^{-1}(G_{j}),  j=1,2.$
By Lemma \ref{le1}, we know that the inverse Laplace transform $g_{2}(t)$ of $G_{2}(s)$ is bounded by $\frac{t^{n\lambda-1}}{\Gamma(n\lambda)}$.
By Theorem \ref{co}, we know that $g_{1}(t)=K_{\frac{2}{n}}^{2}(x,y,t)$ is bounded by $1$ for any $t\in \mathbb{R}$.
Using the convolution formula (\ref{lc}) again, then setting $t=1$, we have
\begin{eqnarray*}|K_{\frac{2}{n}}^{m}(x,y)|&=& \Gamma(n\lambda+1)\biggl|\int_{0}^{1}g_{1}(1-\tau) g_{2}(\tau)d\tau \biggr|
\\& \le&  \Gamma(n\lambda+1) \int_{0}^{1}\frac{\tau^{n\lambda-1}}{\Gamma(n\lambda)}d\tau
\\&=&\frac{ \Gamma(n\lambda+1)}{\Gamma(n\lambda) n\lambda}\\&=&1.\end{eqnarray*}
\end{proof}
\begin{remark}
With this result, valid for both even and odd dimension, we could get the bound of the $(\kappa, a)$-generalized Fourier kernel for any reduced root system with positive multiplicity function $\kappa$ and some $\alpha$.
 Theorem \ref{th7}  greatly extends the applicability of the uncertainty principle and generalized translation operator in \cite{J1} and \cite{Go}.
\end{remark}
\subsection{Integral expression of the kernel for arbitrary $a>0$}
In Theorem \ref{th7}, we have shown that the Fourier kernel $K_{\frac{2}{n}}^{m}(x,y)$ when $m\ge 2$ is the Laplace convolution of the Fourier kernel when $m=2$ and the function $f_{n, k}(t)$ in Lemma \ref{le1}. In this subsection we give the integral expression of the Fourier kernel of $K_{a}^{m}(x,y)$ for $m\ge 2$ and $a>0$.

For general $a>0$ and $m\ge 2$, the Fourier kernel in the Laplace domain can be written as
\begin{eqnarray*}
&&\mathcal{L}(K_{a}^{m}(x,y,t))\nonumber\\&=&2^{2\lambda/a}\Gamma\biggl(\frac{2\lambda+a}{a}\biggr)\frac{1}{r}\biggl(\frac{1}{R}\biggr)^{2\lambda/a}\frac{1-u_{R}^{2}}{(1-2\xi u_{R}+u_{R}^2)^{\lambda+1}} \nonumber\\&=&2^{2\lambda/a}\Gamma\biggl(\frac{2\lambda+a}{a}\biggr)\frac{1}{r}\biggl(\frac{r-s}{(\frac{2}{a}z^{a/2})^2}\biggr)^{2\lambda/a}
\frac{1-u_{R}^{2}}{((u_{R}-e^{i\varrho})(u_{R}-e^{-i\varrho}))^{\lambda+1}},\end{eqnarray*}
where $u_{R}=e^{\frac{-i\pi}{a}} (\frac{2z^{a/2}}{aR})^{2/a}$, $r=\sqrt{s^{2}+(\frac{2}{a}z^{a/2})^{2}}$, $R=s+r$ and $\xi=\frac{e^{i\varrho}+e^{-i\varrho}}{2}$.

It is possible to give an integral expression of this kernel in terms of the generalized Mittag-Leffler function. We give the definition and its Laplace transform here, see also Chapter 2 in \cite{MH}.
\begin{definition}
\label{DefML} The generalized Mittag-Leffler function is defined by
\[E_{\epsilon,\gamma}^{\delta}(z):=\sum_{n=0}^{\infty}\frac{(\delta)_{n}z^{n}}{\Gamma(\epsilon n+\gamma)n!},\]
where $\epsilon, \gamma, \delta \in \mathbb{C}$ with ${\rm Re}\, \epsilon >0.$ For $\delta=1,$ it reduces to the  Mittag-Leffler function.
\end{definition}
The Laplace transform of the generalized  Mittag-Leffler function is
\[\mathcal{L}(t^{\gamma-1}E^{\delta}_{\epsilon, \gamma}(bt^{\epsilon}))=\frac{1}{s^{\gamma}}\frac{1}{(1-bs^{-\epsilon})^{\delta}}\]
where ${\rm Re}\, \epsilon>0$, ${\rm Re}\, \gamma>0$, ${\rm Re}\, s>0$ and $s>|b|^{1/({\rm Re}\, \epsilon)}$, see \cite{MH}.

 Now, we give the integral expression of the $(0, a)$-generalized Fourier kernel as follows.
\begin{theorem}\label{ga} Let $b_{\pm}=e^{\pm i\varrho}e^{i\pi/a}(\frac{2}{a})^{2/a}z$ and
\[
h(t)=z^{-2(\lambda+1)}\int_{0}^{t}\zeta^
{\frac{2}{a}(\lambda+1)-1}E^{\lambda+1}_{\frac{2}{a}, \frac{2}{a}(\lambda+1)}(b_{+}\zeta^{\frac{2}{a}})
(t-\zeta)^
{\frac{2}{a}(\lambda+1)-1}E^{\lambda+1}_{\frac{2}{a}, \frac{2}{a}(\lambda+1)}(b_{-}(t-\zeta)^{\frac{2}{a}})d\zeta
.\]
Then for $a>0$ and $m\ge 2$,  the kernel of the $(0,a)$-generalized Fourier transform is
\begin{eqnarray*}K_{a}^{m}(x,y)&=&c^{m}_{a}\int_{0}^{1}\biggl((1+2\tau)^{-\frac{\lambda}{a}}
J_{\frac{2\lambda}{a}}\biggl(\frac{2}{a}z^{a/2}\sqrt{1+2\tau}\biggr)\\&&-e^{-i\frac{2\pi}{a}}(1+2\tau)^{-\frac{\lambda+2}{a}}
J_{\frac{2\lambda+4}{a}}\biggl(\frac{2}{a}z^{a/2}\sqrt{1+2\tau}\biggr)\biggr)h(\tau)d\tau.
\end{eqnarray*}
with $c^{m}_{a}=2^{-(2\lambda+4)/a}\Gamma\biggl(\frac{2\lambda+a}{a}\biggr)e^{-i\frac{2\pi(\lambda+1)}{a}}a^{4(\lambda+1)/a}$.
\end{theorem}
\begin{proof} Denote $\mathcal{L}(K_{a}^{m}(x,y,t))=H_{1}(s)H_{2}(s)$ where \begin{eqnarray*}H_{1}(s)&=&\frac{1}{(u_{R}-e^{i\varrho})^{\lambda+1}}\cdot\frac{1}{(u_{R}-e^{-i\varrho})^{\lambda+1}}, \\H_{2}(s)&=&2^{2\lambda/a}\Gamma\biggl(\frac{2\lambda+a}{a}\biggr)\frac{1}{r}\biggl(\frac{1}{R}\biggr)^{2\lambda/a}(1-u_{R}^{2}).\end{eqnarray*}
By direct computation, we have
\[H_{1}(s)=e^{-i\frac{2\pi(\lambda+1)}{a}}\biggl(\biggl(\frac{a}{2}\biggr)^{2/a}z^{-1}\biggr)^{2(\lambda+1)}\biggl[\frac{1}{(\varpi^{2/a}-b_{+})^{\lambda+1}}\cdot\frac{1}{(\varpi^{2/a}-b_{-})^{\lambda+1}}\biggr]\]
with $\varpi=r-s$.

Using the generalized  Mittag-Leffler function, we have
\begin{eqnarray} \label{leff}\mathcal{L}^{-1}\biggl(\frac{1}{(s^{2/a}-b)^{\lambda+1}}\biggr)=t^
{\frac{2}{a}(\lambda+1)-1}E^{\lambda+1}_{\frac{2}{a}, \frac{2}{a}(\lambda+1)}(bt^{\frac{2}{a}}).\end{eqnarray}
Now by the inverse Laplace transform formula from \cite{PB}
\begin{eqnarray}\label{lb}\mathcal{L}^{-1}\biggl(\frac{(\sqrt{s^2+a^2}-s)^{\nu}}{\sqrt{s^2+a^2}}F(\sqrt{s^2+a^2}-s)\biggr)=(a^2t)^{\nu/2}\int_{0}^{t}(t+2\tau)^{-\nu/2}J_{\nu}(a\sqrt{t^2+2\tau t})f(\tau)d\tau\end{eqnarray}
where $\mathcal{L}(f(t))=F(s),$ ${\rm Re} \,\nu>-1$ and ${\rm Re} \, s>|{\rm Im}\, a|$ and the Laplace convolution formula (\ref{lc}), we get the result. \end{proof}
 \section{Dunkl kernel associated to the dihedral group}
 \label{dunklFT}

\subsection{Integral expression of the kernel}
The dihedral group $I_{k}$ is the group of symmetries of the regular $k$-gon.   We use complex coordinates $z_{0}=x+iy$ and identify $\mathbb{R}^{2}$ with $\mathbb{C}$. For a fixed $k$ and $j\in\{0,1,\cdots,k-1\}$, the rotations in $I_{k}$ consist of $z_{0}\rightarrow z_{0}e^{2ij\pi/k}$ and the reflections in $I_{k}$ consist of $z_{0}\rightarrow \bar{z}_{0}e^{2ij\pi/k}$. In particular, we have $I_{1}=\mathbb{Z}_{2}$ and $I_{2}=\mathbb{Z}_{2}^{2}$. The weight function associated with $I_{2k}$  and $\kappa=(\alpha,\beta)$ is given by
\[\upsilon_{\kappa}(z_{0})=\biggl|\frac{z_{0}^{k}-\overline{z_{0}}^{k}}{2i}\biggr|^{2\alpha}\biggl|\frac{z_{0}^{k}+\overline{z_{0}}^{k}}{2}\biggr|^{2\beta}.
\]
The weight function $\upsilon_{\kappa}(z_{0})$ associated with the group $I_{k}$, when $k$ is an odd integer, is the same as the weight function $\upsilon_{(\alpha,\beta)}(z_{0})$ associated with the group $I_{2k}$ with $\beta=0$, i.e.
\[\upsilon_{\kappa}(z_{0})=\biggl|\frac{z_{0}^{k}-\overline{z_{0}}^{k}}{2i}\biggr|^{2\alpha}.\]
We also put $P_{j}(G;x,y)$ the reproducing kernel of $\mathcal{H}_{j}(\upsilon_{\kappa})$ and by $P(G;x,y)$ the Poisson kernel, see (\ref{pd}) and (\ref{pois}).
We denote by \[d\mu_{\gamma}(w)=c_{\gamma}(1+w)(1-w^2)^{\gamma-1}dw\] with $c_{\gamma}=[B(\frac{1}{2}, \gamma)]^{-1}.$  It was proved that finding a closed formula of the Poisson kernel which reproduces any $h$-harmonic in the disk reduces to the cases $k=1$ and $k=2$, see \cite{D, DX}.
\begin{theorem}\cite{DX}
\label{DunklPoisson}
For each weight function $\upsilon_{\kappa}(z)$ associated with the group $I_{2k}$, the Poisson kernel is given by
\[P(I_{2k}; z_{1}, z_{2})=\frac{1-|z_{1}|^{2}|z_{2}|^{2}}{1-|z_{1}^{k}|^{2}|\overline{z_{2}}^{k}|^{2}} \frac{|1-z_{1}^{k}\overline{z_{2}}^{k}|^{2}}{|1-\overline{z_{1}}z_{2}|^{2}} P(I_{2}; z_{1}^{k}, z_{2}^{k}),\]
where the Poisson kernel $P(I_{2}; z_{1}, z_{2} )$ associated with $\upsilon_{\kappa}(x+iy)=|y|^{2\alpha}|x|^{2\beta}$ is given by
\begin{eqnarray*} P(I_{2}; z_{1}, z_{2})&=&
\int_{-1}^{1}\int_{-1}^{1}\frac{1-|z_{1}z_{2}|^{2}}{[1-2({\rm Im} \,  z_{1})({\rm Im} \,  z_{2})u -2({\rm Re} \, z_{1}) ({\rm Re} \, z_{2}) v+|z_{1}z_{2}|^{2} ]^{\alpha+\beta+1}}d\mu_{\alpha}(u)d\mu_{\beta}(v).
\end{eqnarray*}
For each weight function $\upsilon_{\kappa}(z)$ associated with  odd-$k$ dihedral group $I_{k}$, the Poisson kernel is given by
\[P(I_{k};z_{1}, z_{2})=\frac{1-|z_{1}|^2|z_{2}|^2}{1-|z_{1}^{k}|^2|\overline{z_{2}}^{k}|^2}\frac{|1-z_{1}^{k}\overline{z_{2}}^{k}|^{2}}{|1-z_{1}\overline{z_{2}}|^2}P(I_{1};z_{1}^{k},z_{2}^{k})
\]
where the Poisson kernel $P(I_{1}; z_{1}, z_{2}) $ associated with $\upsilon_{\kappa}(x+iy)=|y|^{2\alpha}$  is given by
\[P(I_{1};z_{1},z_{2})=\int_{-1}^{1}\frac{1-|z_{1}z_{2}|^2}{(1-2({\rm Im}\, z_{1})({\rm Im}\, z_{2})u-2({\rm Re}\, z_{1})({\rm Re} \,z_{2})+|z_{1}z_{2}|^2)^{\alpha+1}}d\mu_{\alpha}(u).
\]
\end{theorem}

In the following, we write $z_{1}=|z_{1}|\omega, z_{2}=|z_{2}|\eta \in \mathbb{C}$ and $b=|z_{1}||z_{2}|$. Based on the $\mathfrak{sl}_{2}$  relation of $\Delta_{\kappa}, |x|^2$ and the Euler operator,   an orthonormal basis of $L^{2}(\mathbb{R}^{m}, \upsilon_{\kappa}(x)dx)$ for the general Dunkl case and a series expansion of the Dunkl kernel was constructed in \cite{Said, SKO}. In particular,
 the Dunkl kernel $E_{\kappa}(z_{1}, z_{2})=B_{\kappa, 2}(x,y)$ associated with the dihedral group $I_{k}$ has the following series expansion (see also Theorem \ref{the1})
\begin{eqnarray}\label{fe1}
E_{\kappa}(z_{1}, z_{2})= 2^{\langle \kappa \rangle}\Gamma(\langle \kappa\rangle +1)\sum_{j=0}^{\infty}(-i)^{j} b^{-\langle \kappa \rangle} J_{j+\langle \kappa \rangle}(b) P_{j}(I_{k};\omega,\eta)
\end{eqnarray}
with  \[\langle \kappa \rangle=\left\{
                                                                                  \begin{array}{ll}
                                                                                    (\alpha+\beta)k/2, & \hbox{when $k$ is even;} \\
                                                                                    k\alpha, & \hbox{ when $k$ is odd.}
                                                                                  \end{array}
                                                                                \right.\] We introduce an auxiliary  variable $t$  in (\ref{fe1}) as follows \begin{eqnarray*}
E_{\kappa}(z_{1}, z_{2}, t)= 2^{\langle \kappa \rangle}\Gamma(\langle \kappa\rangle +1)\sum_{j=0}^{\infty}(-i)^{j} b^{-\langle \kappa \rangle} J_{j+\langle \kappa \rangle}(bt) P_{j}(I_{k};\omega,\eta).
\end{eqnarray*}
Then fixing $z_{1}, z_{2}\in \mathbb{C}$, we take the Laplace transform with respect to $t$. Using (\ref{l1}),
$r=(s^2+b^2)^{1/2}$ and $R=s+r$,  for ${\rm Re} \,s$ big enough, we have
\begin{eqnarray*}
\mathcal{L}(E_{\kappa}(z_{1}, z_{2}, t))\nonumber&=&\frac{2^{\langle \kappa \rangle}\Gamma(\langle \kappa\rangle +1)}{rR^{\langle \kappa\rangle}}\sum_{j=0}^{\infty} \biggl(\frac{-ib}{R}\biggr)^{j} P_{j}(I_{k}; \omega,\eta)\\&=&\frac{2^{\langle \kappa \rangle}\Gamma(\langle \kappa\rangle +1)}{rR^{\langle \kappa \rangle}}P\biggl(I_{k}; \omega, \frac{-ib}{R}\eta\biggr)
\end{eqnarray*}
where $P(I_{k}; \omega, z_{0}\eta),$ $|z_{0}|<1$ is the analytic continuation of the Poisson kernel $P(I_{k}; \omega, b\eta)$ obtained by acting with the intertwining operator $V_{\kappa}$ on $x$ on both sides of (\ref{ac2}).

In order to get the integral expression of the Dunkl kernel, we first denote and simplify
 \begin{eqnarray*}f_{I_{2k}}(s)&=&\frac{2^{k(\alpha+\beta)}}{rR^{k(\alpha+\beta)}}\frac{1-(\frac{-ib}{ R})^2}{1-(\frac{-ib}{ R})^{2k}}\frac{1-2(\frac{-ib}{ R})^{k}{\rm Re}\,(\omega^{k}\bar{\eta}^{k})+(\frac{-ib}{ R})^{2k}}{1-2{\rm Re}\,(\omega\bar{\eta})(\frac{-ib}{ R})+(\frac{-ib}{ R})^2}\\
&&\times \frac{1-(\frac{-ib}{ R})^{2k}}{(1-2(\frac{-ib}{ R})^{k}(({\rm Im}\, \omega^{k})({\rm Im} \,\eta^{k})u+({\rm Re}\, \omega^{k})({\rm Re} \, \eta^{k})v)+(\frac{-ib}{ R})^{2k})^{\alpha+\beta+1}}\\
&=& \frac{2^{k(\alpha+\beta)}}{r}\frac{\displaystyle \biggl(R+\frac{b^2}{R}\biggr)\biggl(R^{k}-2(-ib)^{k}{\rm Re}\,(\omega^{k}\bar{\eta}^{k})+\displaystyle \frac{(-ib)^{2k}}{R^{k}}\biggr)}{\displaystyle R-2(-ib){\rm Re}\,(\omega\bar{\eta})+\frac{(-ib)^2}{R}}
\\&&\times\frac{1}{\displaystyle \biggl(R^{k}-2(-ib)^{k}(({\rm Im} \,\omega^{k})({\rm Im} \,\eta^{k})u+({\rm Re}\, \omega^{k})({\rm Re}\, \eta^{k})v)+\frac{(-ib)^{2k}}{R^{k}} \biggr)^{\alpha+\beta+1}}
\end{eqnarray*}
and
\begin{eqnarray*}
g_{I_{k}}(s)&=&\frac{2^{k\alpha}}{rR^{k\alpha}}\frac{1-(\frac{-ib}{ R})^2}{1-(\frac{-ib}{ R})^{2k}}\frac{1-2(\frac{-ib}{ R})^{k}\mbox{Re}\,(\omega^{k}\bar{\eta}^{k})+(\frac{-ib}{ R})^{2k}}{1-2\mbox{Re}\,(\omega\bar{\eta})(\frac{-ib}{ R})+(\frac{-ib}{ R})^2}
\\&&\times\frac{1-(\frac{-ib}{ R})^{2k}}{(1-2(\frac{-ib}{ R})^{k}((\mbox{Im}\, \omega^{k})(\mbox{Im} \, \eta^{k})u+(\mbox{Re}\, \omega^{k})(\mbox{Re}\, \eta^{k}))+(\frac{-ib}{ R})^{2k})^{\alpha+1}}
\\&=& \frac{2^{k\alpha}}{r}\frac{\displaystyle \biggl(R+\frac{b^2}{R}\biggr)\biggl(R^{k}-2(-ib)^{k}\mbox{Re}\,(\omega^{k}\bar{\eta}^{k})+\displaystyle \frac{(-ib)^{2k}}{R^{k}}\biggr)}{\displaystyle R-2(-ib)\mbox{Re}\,(\omega\bar{\eta})+\frac{(-ib)^2}{R}}
\\&&\times\frac{1}{\displaystyle \biggl(R^{k}-2(-ib)^{k}((\mbox{Im} \,\omega^{k})(\mbox{Im} \,\eta^{k})u+(\mbox{Re}\, \omega^{k})(\mbox{Re}\, \eta^{k}))+\frac{(-ib)^{2k}}{R^{k}} \biggr)^{\alpha+1}}.\end{eqnarray*}
By $R=s+r=s+\sqrt{s^2+b^2}$ and $\displaystyle\frac{1}{R}=\frac{1}{s+\sqrt{s^2+b^2}}=\frac{\sqrt{s^2+b^2}-s}{b^2}$, we get \begin{eqnarray*}&&R+\frac{b^2}{R}=s+r+b^2\frac{r-s}{b^2}=2r \\&&R+\frac{(-ib)^2}{R}=s+r-b^2\frac{r-s}{b^2}=2s \end{eqnarray*} and
\[
R^{k}+\frac{(-ib)^{2k}}{R^{k}}=(s+r)^{k}+(-1)^{k}(r-s)^{k}=\sum_{j=0}^{k}\binom{k}{j}(1+(-1)^{k+j})s^{j}r^{k-j}
\]
which means that $R^{k}+\frac{(-ib)^{2k}}{R^{k}}$ is always a polynomial in $s$ as $k$ is a positive integer.
 We can  apply Lemma \ref{lem1} because $|({\rm Im} \,\omega^{k})({\rm Im} \,\eta^{k})u+({\rm Re}\, \omega^{k})({\rm Re}\, \eta^{k})v|\le 1$, for $u,v \in [-1, 1]$. Hence $f_{I_{2k}}(s)$ and $ g_{I_{k}}(s)$ have the following factorization,
\begin{lemma}\label{lem3} Let \[A(s,q)=\prod_{l=0}^{k-1}\biggl(s+ib\cos\biggl(\frac{q+2\pi l}{k}\biggr)\biggr),\] \[B(s)=(s+ib{\rm Re}\,(\omega\bar{\eta})).\]
Then $f_{I_{2k}}(s)$ has the following factorization
\begin{eqnarray*}
f_{I_{2k}}(s)&=&\frac{A(s,q(1,1))}{B(s)[A(s,q(u,v))]^{\alpha+\beta+1}}
=\frac{1}{B(s)[A(s,q(u,v))]^{\alpha+\beta}}+\frac{(-ib)^{k}\cos(q(u-1,v-1))}{2^{k-1}B(s)[A(s,q(u,v))]^{\alpha+\beta+1}},
\end{eqnarray*}
 and $g_{I_{k}}(s)$ has the following factorization\begin{eqnarray*}g_{I_{k}}(s)&=&\frac{A(s,q(1,1))}{B(s)[A(s,q(u,1))]^{\alpha+1}}
=\frac{1}{B(s)[A(s,q(u,1))]^{\alpha}}+\frac{(-ib)^{k}\cos(q(u-1,0))}{2^{k-1}B(s)[A(s,q(u,1))]^{\alpha+1}}.
\end{eqnarray*}
where  $q(u, v)=\arccos(({\rm Im} \,\omega^{k})({\rm Im} \,\eta^{k})u+({\rm Re}\, \omega^{k})({\rm Re}\, \eta^{k})v)$.

\end{lemma}
\begin{proof} For the first equality, we only need to show that $q(1, 1)= \arccos({\rm Re}\,(\omega^{k}\bar{\eta}^{k}))$, i.e. \[{\rm Re}\,(\omega^{k}\bar{\eta}^{k})=({\rm Im} \,\omega^{k})({\rm Im} \,\eta^{k})+({\rm Re}\, \omega^{k})({\rm Re}\, \eta^{k})\]
which follows by expanding the left-hand side.
For the second equality, we have used
\[2^{k}A(s,q(u,v))= R^{k}-2(-ib)^{k}((\mbox{Im} \,\omega^{k})(\mbox{Im} \,\eta^{k})u+(\mbox{Re}\, \omega^{k})(\mbox{Re}\, \eta^{k}))+\frac{(-ib)^{2k}}{R^{k}}.\]
\end{proof}
Now, we have our first main result in this section
\begin{theorem}\label{ld1} For the even dihedral group $I_{2k}$, the radial Laplace transform of the Dunkl kernel is
\begin{eqnarray*}
\mathcal{L}(E_{\kappa}(z_{1}, z_{2},t))&=&\Gamma(k(\alpha+\beta) +1)
\int_{-1}^{1}\int_{-1}^{1}f_{I_{2k}}(s)d\mu_{\alpha}(u)d\mu_{\beta}(v).
\end{eqnarray*}
For odd-$k$ dihedral group $I_{k}$, the radial Laplace transform of the Dunkl kernel $E_{\kappa}(z_{1}, z_{2}, t)$ is
\begin{eqnarray*}
\mathcal{L}(E_{\kappa}(z_{1}, z_{2}, t))&=&\Gamma(k\alpha +1)
\int_{-1}^{1}
g_{I_{k}}(s)d\mu_{\alpha}(u).
\end{eqnarray*}
\end{theorem}

For any dihedral group, when the multiplicity function $\kappa$ takes integer values, we know from Lemma \ref{lem3} that $f_{I_{2k}}(s)$ and $g_{I_{k}}(s)$ are rational functions. So then the Dunkl kernel can be obtained by the inverse Laplace transform through partial fraction decomposition using Theorem \ref{ld1} and \ref{pfd}.
\begin{remark} It is known that the Dunkl kernel for positive integer $\kappa$ can in principle be expressed as elementary functions, see  \cite{O1} and \cite{DJ}. However, this is not made concrete there. In \cite{DDY}, the authors use the shift principle of \cite{O1} and act with multiple combinations of the Dunkl operators on the Dunkl Bessel function to derive the Dunkl kernel in the dihedral setting. However, there the Dunkl Bessel function was only known in a few cases. In subsection 4.2, we will give the integral expression of the generalized Bessel function using the Laplace transform. Also, acting multiple combinations of the Dunkl operators turns out  not to be feasible in practice.
\end{remark}

When the multiplicity function $\kappa$ is not integer valued, we can still derive integral formulas for the kernel using Theorem \ref{ld1}.
First denote
\begin{eqnarray*}g_{\alpha}(t, q(u,v))&=&\mathcal{L}^{-1}\biggl(\frac{1}{A(s,q(u,v))^{\alpha}}\biggr)=\mathcal{L}^{-1}\biggl(\frac{1}{\prod_{l=0}^{k-1}\biggl(s+ib\cos\biggl(\frac{q(u,v)+2\pi l}{k}\biggr)\biggr)^{\alpha}}\biggr)
\\&=&\mathcal{L}^{-1}\biggl(\frac{2^{k\alpha}}{
\biggl(R^{k}-2(-ib)^{k}(({\rm Im} \,\omega^{k})({\rm Im} \,\eta^{k})u+({\rm Re}\, \omega^{k})({\rm Re}\, \eta^{k})v)+\frac{(-ib)^{2k}}{R^{k}} \biggr)^{\alpha}}\biggr)
\\&=&\mathcal{L}^{-1}\biggl(\frac{2^{k\alpha-1}e^{i k\alpha\pi}\varpi_{0}^{k\alpha}}{r}\biggl(\frac{b^2}{\varpi_{0}}+\varpi_{0}\biggr)\cdot \frac{1}{(\varpi_{0}^{k}-e^{iq(u,v)}(e^{i\frac{\pi}{2}}b)^{k})^{\alpha}(\varpi_{0}^{k}-e^{-iq(u,v)}(e^{i\frac{\pi}{2}}b)^{k})^{\alpha}}
\biggr)
\end{eqnarray*}
where $\varpi_{0}=r-s$. Using the same method as in Theorem \ref{ga}, by formula (\ref{leff}) and (\ref{lb}),  we have
\begin{eqnarray*}g_{\alpha}(t, q(u,v))&=&2^{k\alpha-1}e^{ik\alpha\pi}b^{k\alpha+1}\int_{0}^{t}\biggl[J_{k\alpha-1}(b\sqrt{t^2-2\tau t})\\&&+t(t+2\tau)^{-1}J_{k\alpha+1}(b\sqrt{t^2+2\tau t})\biggr]t^{\frac{k\alpha-1}{2}}(t+2\tau)^{-\frac{k\alpha-1}{2}}\tilde{h}_{\alpha}(\tau)d\tau,
\end{eqnarray*}
where $\tilde{h}_{\alpha}(t)$ is the convolution of two  generalized Mittag-Leffler functions,
\begin{eqnarray*}
\tilde{h}_{\alpha}(t)=\int_{0}^{t}\zeta^
{k\alpha-1}E^{\alpha}_{k, k\alpha}(e^{iq(u,v)}(e^{i\frac{\pi}{2}}b)^{k}\zeta^{k})
(t-\zeta)^{k\alpha-1}E^{\alpha}_{k, k\alpha}(e^{-iq(u,v)}(e^{i\frac{\pi}{2}}b)^{k}(t-\zeta)^{k})d\zeta.
\end{eqnarray*}
Now, by the convolution formula (\ref{lc}), we have
\begin{theorem} \label{m1} Let $a_{u,v}^{l}$ be the $k+1$ roots of $B(s)A(s,q(u,v))$, i.e. $a^{l}_{u,v}=\displaystyle -ib\cos\biggl(\frac{q+2\pi l}{k}\biggr)$, $l=0,\cdots, k-1$ and $a^{k}_{u,v}=-ib {\rm Re}\,(\omega \bar{\eta})$.  Then for each dihedral group $I_{2k}$ and positive multiplicity function $\kappa$, the Dunkl kernel is given by
\begin{eqnarray*}
E_{\kappa}(z_{1}, z_{2})&=&\Gamma(k(\alpha+\beta)+1)
\int_{-1}^{1}\int_{-1}^{1}\int_{0}^{1}\biggl[\sum_{l=0}^{k}\frac{A(s,q(1,1))(s-a_{u,v}^{l})}{B(s)A(s,q(u,v))}\biggl|_{s=a_{u,v}^{l}}e^{a^{l}_{u,v}(1-\tau)}\biggr]
\\&&g_{\alpha+\beta}(\tau, q(u,v))
 d\tau d\mu_{\alpha}(u)d\mu_{\beta}(v).\end{eqnarray*}
For each odd-$k$ dihedral group $I_{k}$ and positive multiplicity function $\kappa$, the Dunkl kernel is
\[E_{\kappa}(z_{1}, z_{2})=\Gamma(k\alpha +1)\int_{-1}^{1}\int_{0}^{1}\biggl[\sum_{l=0}^{k}\frac{A(s,q(1,1)(s-a_{u,1}^{l}))}{B(s)A(s,q(u,1))}\biggl|_{s=a_{u,1}^{l}}e^{a^{l}_{u,1}(1-\tau)}\biggr]g_{\alpha}(\tau, q(u,1))d\tau d\mu_{\alpha}(u),\]
where  $q(u, v)=\arccos(({\rm Im} \,\omega^{k})({\rm Im} \,\eta^{k})u+({\rm Re}\, \omega^{k})({\rm Re}\, \eta^{k})v)$.
\end{theorem}
\begin{proof} We only prove  the odd dihedral group $I_{k}$ cases. We write $g_{I_{k}}$ as

\begin{eqnarray}\label{fi}g_{I_{k}}(s)
=\frac{A(s,q(1,1))}{B(s)[A(s,q(u,1))]}\cdot\frac{1}{[A(s,q(u,1))]^{\alpha}}.\end{eqnarray}
The inverse Laplace transform of the second factor on the right-hand side of (\ref{fi}) is $g_{\alpha}(t, q(u,1))$. The first factor on the right-hand side of  (\ref{fi}) is inversed by partial fraction decomposition. Then by the Laplace convolution formula (\ref{lc}), we get the result.

\end{proof}

Using the second equality in Lemma \ref{lem3}, the integral expression of the Dunkl kernel also reduces to the integral expression of $f_{n,k}(t)$ in Lemma \ref{le1}. Indeed, put
\[h_{\alpha}(t, q(u,v))=\mathcal{L}^{-1}\biggl(\frac{1}{B(s)A(s,q(u,v))^{\alpha}}\biggr)=\mathcal{L}^{-1}\biggl(\frac{1}{(s+ib{\rm Re}\,(\omega\bar{\eta}))\prod_{l=0}^{k-1}\biggl(s+ib\cos\biggl(\frac{q(u,v)+2\pi l}{k}\biggr)\biggr)^{\alpha}}\biggr),\]
which is the convolution  of $g_{\alpha}(t, q(u,v))$ and $e^{-ib{\rm Re}\,(\omega\bar{\eta})}$.
Then we have
\begin{theorem}\label{m2}
 For each dihedral group $I_{2k}$ and positive multiplicity function $\kappa$, the Dunkl kernel is given by
\begin{eqnarray*}
E_{\kappa}(z_{1}, z_{2})&=&\Gamma(k(\alpha+\beta)+1)
\int_{-1}^{1}\int_{-1}^{1}[h_{\alpha+\beta}(1, q(u,v))\\&&+2^{1-k}(-ib)^{k}\cos(q(u-1,v-1))h_{\alpha+\beta+1}(1, q(u,v))]d\mu_{\alpha}(u)d\mu_{\beta}(v).\end{eqnarray*}
For each odd-$k$ dihedral group $I_{k}$ and positive multiplicity function $\kappa$, the Dunkl kernel is
\[E_{\kappa}(z_{1}, z_{2})=\Gamma(k\alpha +1)\int_{-1}^{1}[h_{\alpha}(1, q(u,1))+2^{1-k}(-ib)^{k}\cos(q(u-1,0))h_{\alpha+1}(1, q(u,1))]d\mu_{\alpha}(u),\]
where  $q(u, v)=\arccos(({\rm Im} \,\omega^{k})({\rm Im} \,\eta^{k})u+({\rm Re}\, \omega^{k})({\rm Re}\, \eta^{k})v)$.
\end{theorem}

Let us now look at a few special cases.
When $k=1$ and any positive $\alpha$, $g_{I_{1}}(s)$ becomes
\begin{eqnarray}\label{k1}g_{I_{1}}(s)&=& \frac{1}{(s+ib((\mbox{Im}\, \omega)(\mbox{Im}\, \eta)u+(\mbox{Re}\, \omega)(\mbox{Re}\, \eta)))^{\alpha+1}}.\end{eqnarray}
We take the inverse Laplace transform of (\ref{k1}) and set $t=1$,   then we reobtain the Dunkl kernel for $I_{1}$, which is
\[E_{ \kappa}(z_{1},z_{2})=\int_{-1}^{1}e^{-\displaystyle i(u \mbox{Im}\,z_{1}\mbox{Im}\,z_{2}+\mbox{Re}\,z_{1} \mbox{Re}\, z_{2} )}d\mu_{\alpha}(u).
\]
It coincides with the known result of the integral representation of the intertwining operator of the rank $1 $ case, for $\mbox{Re} \,\alpha>0$,
\[V_{\alpha}p(x)=\int_{-1}^{1}p(xu)d\mu_{\alpha}(u),
\]
which can be found in \cite{DX}.
Similarly, we reobtain the Dunkl kernel for $I_{2}$, which is
\[E_{\kappa}(z_{1},z_{2})=\int_{-1}^{1}\int_{-1}^{1}e^{-\displaystyle i(u\mbox{Im}\,z_{1}\mbox{Im}\,z_{2}+v\mbox{Re}\,z_{1} \mbox{Re}\, z_{2} )}d\mu_{\alpha}(u)d\mu_{\beta}(v)
\]
which coincides with the result obtained using the intertwining operator for $\mathbb{Z}_{2}^{2}$.

For the dihedral group $I_{3}$ and $I_{6}$, we can get the integral expression of the Dunkl kernels by (\ref{I3}) as both of them are related to the function $f_{3, k}(t)$.

For the dihedral group $I_{4},$  we have
\begin{eqnarray*}f_{I_{4}}(s)=\frac{s^2+b^2\biggl(\frac{1+{\rm Re}\,\omega^2 \bar{\eta}^2}{2}\biggr)}{(s+ib{\rm Re}\, \omega\bar{\eta} )\biggl(s^2+b^2\biggl(\frac{1+({\rm Im}\,\omega^2 )({\rm Im} \,\eta^2) u+({\rm Re}\,\omega^2) ({\rm Re}\, \eta^2) v}{2}\biggr)\biggr)^{\alpha+\beta+1}}.
\end{eqnarray*}
We take the inverse Laplace transform and set $t=1$. We get the Dunkl kernel for $I_{4}$, using Theorem \ref{m1},
\begin{eqnarray*}
E_{\kappa}(z_{1}, z_{2})&=&\frac{\sqrt{\pi}\Gamma(2(\alpha+\beta)+1)}{2^{\alpha+\beta-1/2}\Gamma(\alpha+\beta)}
\int_{-1}^{1}\int_{-1}^{1}\int_{0}^{1}
\frac{1}{\theta_{2}^{2}-\theta_{3}^{2}}\biggl(e^{-ib\theta_{3}\tau}(\theta_{1}^2-\theta_{3}^{2})
+(\theta_{1}^2-\theta_{2}^{2})\biggl(\frac{i\theta_{3}}{\theta_{2}}\sin(b\theta_{2}\tau)-\cos(b\theta_{2}\tau)\biggr)\biggr)\\&&
(1-\tau)^{\alpha+\beta-1/2}\frac{J_{\alpha+\beta-1/2}(b\theta_{2}(1-\tau) )}{(b\theta_{2})^{\alpha+\beta-1/2}}
 d\tau d\mu_{\alpha}(u) d\mu_{\beta}(v),\end{eqnarray*}
or using Theorem \ref{m2},
\begin{eqnarray*}
E_{\kappa}(z_{1}, z_{2})&=&\frac{\sqrt{\pi}\Gamma(2(\alpha+\beta)+1)}{2^{\alpha+\beta-1/2}\Gamma(\alpha+\beta)}
c_{\alpha}c_{\beta}\int_{-1}^{1}\int_{-1}^{1}\int_{0}^{1}
e^{-ib\theta_{3}\tau}\biggl(
(1-\tau)^{\alpha+\beta-1/2}\frac{J_{\alpha+\beta-1/2}(b\theta_{2}(1-\tau) )}{(b\theta_{2})^{\alpha+\beta-1/2}}\\&&+
\frac{b^{2}(\theta_{1}^2-\theta_{2}^{2})}{2(\alpha+\beta)}(1-\tau)^{\alpha+\beta+1/2}\frac{J_{\alpha+\beta+1/2}(b\theta_{2}(1-\tau) )}{(b\theta_{2})^{\alpha+\beta+1/2}}\biggr)
 d\tau d\mu_{\alpha}(u) d\mu_{\beta}(v),\end{eqnarray*}
where $\theta_{1}=\sqrt{\frac{1+({\rm Re}\,\omega^2 \bar{\eta}^2)}{2}}$, $\theta_{2}=\sqrt{\frac{1+({\rm Im}\,\omega^2 )({\rm Im} \,\eta^2) u+({\rm Re}\,\omega^2)( {\rm Re}\, \eta^2 )v}{2}}$ and $\theta_{3}={\rm Re}\, \omega\bar{\eta}$.

\begin{remark} The kernel of the  $(\kappa,a)$-generalized Fourier transform with dihedral symmetry can be obtained similarly.
\end{remark}
\subsection{Dunkl Bessel function}

Following \cite{DX},
we define the Dunkl Bessel function by
\[D_{\kappa}(z_{1}, z_{2})=\frac{1}{|I_{k}|}\sum_{g\in I_{k}}E_{\kappa}(z_{1}, g\cdot z_{2}).\]
Let $z_{1}=|z_{1}|e^{i\phi_{1}}$, $z_{2}=|z_{2}|e^{i\phi_{2}}$, $\phi_{1}, \phi_{2} \in [1,\pi/2k]$ and $b=|z_{1}||z_{2}|$. Then the Dunkl Bessel function associated to $I_{2k}, k\ge 2$ is given by (see \cite{Dn})
\[D_{\kappa}(|z_{1}|, \phi_{1},  |z_{2}|, \phi_{2} )=c_{k,\kappa}\biggl(\frac{2}{b}\biggr)^{\langle \kappa \rangle}\sum_{j= 0}^{\infty} i^{2kj+\langle \kappa\rangle } J_{2kj+\langle \kappa\rangle}(b)p_{j}^{\alpha-1/2,\beta-1/2}(\cos(2k\phi_{1})) p_{j}^{\alpha-1/2,\beta-1/2}(\cos(2k\phi_{2}))  \]
where $p_{j}^{\alpha-1/2,\beta-1/2}$ is the $j$-th orthonormal Jacobi polynomial of parameters $(\alpha-1/2,\beta-1/2)$ and \[c_{k,\kappa}=2^{\alpha+\beta}\frac{\Gamma(\langle \kappa\rangle +1)\Gamma(\alpha+1/2)\Gamma(\beta+1/2) }{\Gamma(\alpha+\beta+1)}.\]

With the Dijksma-Koornwinder product formula for the Jacobi polynomial, the Dunkl Bessel function becomes
\begin{eqnarray}\label{B1}D_{\kappa}(|z_{1}|, \phi_{1},  |z_{2}|, \phi_{2} )&=&\Gamma(\langle \kappa\rangle +1)\int_{-1}^{1}\int_{-1}^{1}\biggl(\frac{2}{b}\biggr)^{\langle \kappa \rangle}\sum_{j= 0}^{\infty}\frac{(2j+\alpha+\beta)}{\alpha+\beta} i^{2kj+\langle \kappa\rangle } J_{2kj+\langle \kappa\rangle}(b)\nonumber\\&&C_{2j}^{\alpha+\beta}(z_{k\phi_{1},k\phi_{2}}(u,v))\mu^{\alpha}(du)\mu^{\beta}(dv) \end{eqnarray}
where $\mu^{\alpha}$  is the symmetric beta probability measure
\[\mu^{\alpha} (du)=\frac{\Gamma(\alpha+1/2)}{\sqrt{\pi}\Gamma(\alpha)}(1-u^2)^{\alpha-1}du, \quad \alpha>-1,\]
and
 \[z_{\phi_{1},\phi_{2}}(u,v)=u\cos \phi_{1}\cos \phi_{2}+v\sin\phi_{1}\sin \phi_{2},\]
and $C_{2j}^{\alpha}(x)$ the Gegenbauer polynomial.
Now the integrand of (\ref{B1}) equals $\displaystyle\frac{f_{2k}^{+}+f_{2k}^{-}}{2}$ with
\begin{eqnarray*}f_{2k}^{\pm}(b, \xi)=\Gamma(k(\alpha+\beta)+1)\biggl(\frac{2}{b}\biggr)^{k(\alpha+\beta)}\sum_{j=0}^{\infty}\frac{(j+\alpha+\beta)}{\alpha+\beta} (\pm1)^{j} e^{i\frac{\pi}{2}k(j+\alpha+\beta) } J_{k(j+\alpha+\beta)}(b)C_{j}^{\alpha+\beta}(z_{k\phi_{1},k\phi_{2}}).
\end{eqnarray*}
As before, we introduce an auxiliary variable $t$ in the series
\begin{eqnarray*}f_{2k}^{\pm}(b, \xi, t)=\Gamma(k(\alpha+\beta)+1)\biggl(\frac{2}{b}\biggr)^{k(\alpha+\beta)}\sum_{j=0}^{\infty}\frac{(j+\alpha+\beta)}{\alpha+\beta} (\pm1)^{j} e^{i\frac{\pi}{2}k(j+\alpha+\beta) } J_{k(j+\alpha+\beta)}(bt)C_{j}^{\alpha+\beta}(z_{k\phi_{1}, k\phi_{2}}).
\end{eqnarray*}
and take the Laplace transform term by term. This yields
\begin{eqnarray}\label{bes}
\mathcal{L}(f_{2k}^{\pm})&=&\Gamma(k(\alpha+\beta)+1)\frac{(2e^{i\frac{\pi}{2}})^{k(\alpha+\beta)}}{r}\frac{R^{k}-\frac{(-1)^{k}b^{2k}}{R^{k}}}{(R^{k}-2(\pm (ib)^{k})z_{k\phi_{1},k\phi_{2}}+\frac{(-1)^{k}b^{2k}}{R^{k}} )^{\alpha+\beta+1}}
\nonumber\\&=&\Gamma(k(\alpha+\beta)+1)\frac{(2e^{i\frac{\pi}{2}})^{k(\alpha+\beta)}}{r}\frac{(r+s)^{k}-(-1)^{k}(r-s)^{k}}{((r+s)^{k}-2(\pm (ib)^{k})z_{k\phi_{1},k\phi_{2}}+(-1)^{k}(r-s)^{k} )^{\alpha+\beta+1}}
\end{eqnarray}
where $r=\sqrt{r^2+b^2}$, $R=s+r$.
Comparing (\ref{bes}) with (\ref{rl3}),  and using the same method as in Theorem \ref{th7}, we get $|f_{2k}^{\pm}|\le 1$.
Then we have
\[|D_{\kappa}(z_{1},z_{2})|=\biggl|\int_{-1}^{1}\int_{-1}^{1} \frac{f_{2k}^{+}+f_{2k}^{-}}{2} \mu^{\alpha}(du)\mu^{\beta}(dv)\biggr|\le 1\]
because $\int_{-1}^{1}\int_{-1}^{1}\mu^{\alpha}(du)\mu^{\beta}(dv)=1$, giving an alternative and direct proof of the boundedness of the Dunkl Bessel function. Also, using (\ref{bes}) and (\ref{rl3}), it is now in principle possible to find an integral expression for the Dunkl Bessel function. We illustate this for
 the dihedral group $I_{4}$. In that case, we have
\begin{eqnarray*}
\mathcal{L}(f_{4}^{\pm})&=&\Gamma(2(\alpha+\beta)+1)\frac{(2e^{i\frac{\pi}{2}})^{2(\alpha+\beta)}}{r}\frac{(r+s)^{2}-(-1)^{2}(r-s)^{2}}{((r+s)^{2}-2(\pm (ib)^{2})z_{2\phi_{1}, 2\phi_{2}}+(-1)^{2}(r-s)^{2} )^{\alpha+\beta+1}}
\\&=&\Gamma(2(\alpha+\beta)+1)e^{i\pi (\alpha+\beta)}\displaystyle\frac{s}{\biggl(s^2+b^2\biggl(\frac{1-(\pm z_{2\phi_{1},2\phi_{2}})}{2}\biggr)\biggr)^{\alpha+\beta+1}}.
\end{eqnarray*}
Using the inverse Laplace transform formula (\ref{nf1}),
we have, after evaluating at $t=1$,
\begin{eqnarray*}f_{4}^{+}+f_{4}^{-}&=&e^{i\pi(\alpha+\beta)}\frac{\sqrt{\pi}\Gamma(2(\alpha+\beta)+1)}
{\Gamma(\alpha+\beta+1)2^{\alpha+\beta+1/2}}\biggl(\frac{J_{\alpha+\beta-1/2}(b_{1})}{b_{1}^{\alpha+\beta-1/2}}
+\frac{J_{\alpha+\beta-1/2}(b_{2})}{b_{2}^{\alpha+\beta-1/2}}\biggr)\\
&=&e^{i\pi(\alpha+\beta)}2^{\alpha+\beta-1/2}
\Gamma(\alpha+\beta+1/2)\biggl(\frac{J_{\alpha+\beta-1/2}(b_{1})}{b_{1}^{\alpha+\beta-1/2}}
+\frac{J_{\alpha+\beta-1/2}(b_{2})}{b_{2}^{\alpha+\beta-1/2}}\biggr)
\end{eqnarray*}
where $b_{1}=\biggl(b\sqrt{\frac{1-z_{2\phi_{1},  2\phi_{2}}}{2}}\biggr)$, $b_{2}=\biggl(b\sqrt{\frac{1+ z_{2\phi_{1},  2\phi_{2}}}{2}}\biggr)$. In the second equality, we have used the Gauss duplication formula
\[\sqrt{\pi}\Gamma(2v)=2^{2v-1}\Gamma(v)\Gamma(v+1/2).\] Hence for $I_{4}$, the Dunkl  Bessel function is given by

\[D_{\kappa}(z_{1},z_{2})=e^{i\pi(\alpha+\beta)}2^{\alpha+\beta-3/2}
\Gamma(\alpha+\beta+1/2)\int_{-1}^{1}\int_{-1}^{1} \biggl(\frac{J_{\alpha+\beta-1/2}(b_{1})}{b_{1}^{\alpha+\beta-1/2}}
+\frac{J_{\alpha+\beta-1/2}(b_{2})}{b_{2}^{\alpha+\beta-1/2}}\biggr) \mu^{\alpha}(du)\mu^{\beta}(dv).\]
\begin{remark}
When $\alpha+\beta$ is integer, the integral expression of the Dunkl Bessel function associated to $I_{4}$ was obtained in \cite{Dn}. Our result hence extends this result to arbitrary $\alpha, \beta>0$.
\end{remark}
\begin{remark} For  odd dihedral groups, the integral expression of the Dunkl Bessel function is computed in a similar way.
\end{remark}

\

 {\footnotesize{\bf Acknowledgments} \qquad H. De Bie is supported by the UGent BOF starting grant 01N01513. P. Lian is supported by a scholarship from Chinese Scholarship Council (CSC) under the CSC No. 201406120169.}

\end{document}